\documentclass[11pt]{article}

\usepackage{amsmath,amsfonts,amssymb} % for math
\usepackage{multirow}
\usepackage{amsthm} % theorems
\usepackage{enumerate,ulem}
\usepackage{graphicx,xcolor}
\usepackage{tikz}

\usepackage{graphicx,caption,subcaption} % figures and subfigures
\usepackage{fullpage}

\allowdisplaybreaks % otherwise the align environment is stupid.

\begingroup
    \makeatletter
    \@for\theoremstyle:=definition,remark,plain\do{%
        \expandafter\g@addto@macro\csname th@\theoremstyle\endcsname{%
            \addtolength\thm@preskip\parskip
            }%
        }
\endgroup

\newtheorem{thm}{Theorem}[section]
\newtheorem{cor}[thm]{Corollary}

\newtheorem{lemma}[thm]{Lemma}

\newtheorem{claim}[thm]{Claim}

\newtheorem{obs}[thm]{Observation}

\newcommand{\N}{\mathbb{N}}
\newcommand{\powerset}[1]{\operatorname{Pow}(#1)}

\newcommand{\cDP}{\chi_\mathrm{DP}}
\newcommand{\fDP}{f_\mathrm{DP}}

\def\sspace{0.3}
\def\picvxs{
    \draw(0,0)node[graynode](f1){};
    \draw(\sspace,0.3)node[graynode](f2){};
    \draw(\sspace,-0.3)node[graynode](f3){};
    \draw(\sspace*2,0)node[graynode](f4){};
    \draw(\sspace*3,0.3)node[graynode](f5){};
    \draw(\sspace*3,-0.3)node[graynode](f6){};
    \draw(\sspace*4,0)node[graynode](f7){};
}

\newcommand\picpicd{
\picvxs
\draw(\sspace,1)node[blacknode](t1){};
\draw(\sspace*2,1)node[blacknode](t2){};
\draw(\sspace*1.5,1.5)node[blacknode,label=$z'$](t3){};
\foreach \from/\to in {t1/f2,t2/f4,t1/t2,t2/t3,t3/t1,f1/f2,f1/f3,f2/f3,f3/f4,f4/f5,f5/f6,f4/f6,f5/f7,f6/f7}
    \draw (\from) -- (\to);
\foreach \from/\to in {f4/f5,f5/f6,f4/f6,f5/f7,f6/f7}
    \draw[ultra thick] (\from) -- (\to);
\draw (f1) to [out=270,in=270,looseness=1.5] (f7);
\draw (t3) to [out=0,in=90] (f7);

\draw (f2)++(-0.15,0)node {$y$};
\draw (f4)++(0,-0.3)node {$x$};
\draw (f7)++(0.15,0)node  {$z$};

\draw (t1)++(-0.15,0)node {$y'$};
\draw (t2)++(0.15,0)node {$x'$};

\draw (f4)++(0,-1)node {(a)};
}

\newcommand\picpicg{
\picvxs
\draw(\sspace*1.5,1+0.5)node[blacknode,label=above:$y'$](t1){};
\draw(\sspace*2.5,1+0.5)node[blacknode,label=above:$x'$](t2){};
\draw(\sspace*2,1.5-0.5)node[blacknode,label=right:$z'$](t3){};

\foreach \from/\to in {t1/t2,t2/t3,t3/t1,f1/f2,f1/f3,f2/f3,f3/f4,f4/f5,f5/f6,f4/f6,f5/f7,f6/f7,t2/t3,f2/f4}
    \draw (\from) -- (\to);
\foreach \from/\to in {f4/f5,f5/f6,f4/f6,f5/f7,f6/f7}
    \draw[ultra thick] (\from) -- (\to);

\draw (t3) to [out=270,in=90] (f4);
\draw (t1) to [out=180,in=90] (f1);
\draw (t2) to [out=0,in=90] (f7);

\draw (f7)++(0,-0.3)node {$x$};
\draw (f4)++(0,-0.3)node {$z$};
\draw (f1)++(0,-0.3)node {$y$};

\draw (f4)++(0,-1)node {(b)};
}

\title{\vspace{-0.5in} A lower bound on the number of edges in DP-critical graphs. II. Four colors}

\author{
{{Peter Bradshaw}}\thanks{
\footnotesize{University of Illinois at Urbana--Champaign, Urbana, IL 61801, USA.
 E-mail: \texttt {pb38@illinois.edu}.
 Research %%% of this author
is supported in part by NSF RTG grant DMS-1937241.
}}
\and
{{Ilkyoo Choi}}\thanks{
\footnotesize {Hankuk University of Foreign Studies, Yongin-si, Gyeonggi-do, Republic of Korea.
 E-mail: \texttt {ilkyoo@hufs.ac.kr}.
 Research %%% of this author
is supported in part by the Hankuk University of Foreign Studies Research Fund.
}}
\and
{{Alexandr Kostochka}}\thanks{
\footnotesize {University of Illinois at Urbana--Champaign, Urbana, IL 61801, USA.
 E-mail: \texttt {kostochk@illinois.edu}.
 Research %%% of this author
is supported in part by  NSF  Grant DMS-2153507 and by NSF RTG Grant DMS-1937241.
}}
\and{{Jingwei Xu}}\thanks{University of Illinois at Urbana--Champaign, Urbana, IL 61801, USA. 
 E-mail: \texttt{jx6@illinois.edu}. % it should compile now
 Research 
is supported in part by   Campus Research Board Award RB24000 of the University of Illinois Urbana-Champaign.
}}

\begin{document}
\maketitle

\vspace{-0.3in}

\begin{abstract}
A graph $G$ is $k$-critical (list $k$-critical,  DP $k$-critical) if $\chi(G)= k$ ($\chi_\ell(G)= k$, $\cDP(G)= k$) and for every proper subgraph $G'$ of $G$, $\chi(G')<k$ ($\chi_\ell(G')< k$, $\cDP(G')<k$). 
Let $f(n, k)$ ($f_\ell(n, k), \fDP(n,k)$) denote the minimum number of edges in an $n$-vertex
$k$-critical (list $k$-critical, DP $k$-critical) graph. 
The main result of this paper is that if  $n\geq 6$ and $n\not\in\{7,10\}$, then
$$\fDP(n,4)>\left(3 + \frac{1}{5}  \right) \frac{n}{2}.
$$
This is the first bound on $\fDP(n,4)$ that is asymptotically better than the well-known bound  $f(n,4)\geq \left(3 + \frac{1}{13}  \right) \frac{n}{2}$ by Gallai from 1963.
The result also yields a  better bound on $f_{\ell}(n,4)$ than the one known before.
\vspace{0.15in}
 
\end{abstract}

 \section{Introduction}
%Let $[k]$ denote the set $\{1, \ldots, k\}$.
%Let $\N$ denote the set of positive integers.

% Given a vertex pair $u,v \in V(G)$, we write $E_G(u,v)$ for the set of edges joining $u$ and $v$.
% Given subsets $X,Y \subseteq V(G)$, we write $E_G(X,Y)$ for the set of edges in $G$ with an endpoint in both $X$ and $Y$.
% Given a subset $A \subseteq V(G)$, we write $|A|$ for the number of vertices in $A$ and $\|A\|$ for the number of edges in $G[A]$.

 \subsection{Critical graphs for proper coloring}
 Given a multigraph $G$, let $V(G)$ and $E(G)$ denote the vertex set and edge set, respectively, of $G$. 
A {\em (proper) $k$-coloring} of a graph $G$ is a mapping $g\,: \,V(G)\to\{1,\ldots,k\}$ such that $g(u)\neq g(v)$ for each $uv\in E(G)$.
The {\em chromatic number} of $G$,  denoted $\chi(G)$, is the minimum positive integer $k$ for which $G$ has a proper $k$-coloring.
A graph $G$ is {\em $k$-colorable} if $\chi(G)\leq k$. 
 For a positive integer $k$, a graph $G$ is {\em $k$-critical} if $\chi(G)=k$, but every
proper subgraph of $G$ is $(k-1)$-colorable.

Dirac~\cite{1951Dirac,1952DiracAproperty,1952DiracSome,1953Dirac} introduced and studied $k$-critical graphs beginning in 1951.  
In particular, he considered the minimum number $f(n,k)$ of edges in an $n$-vertex $k$-critical graph.
%With this notation, $f(k,k)=\binom{k}{2}$
%and $f(k+1,k)$ is not well defined. 
The bounds on $f(n,k)$ turned out to be quite helpful in studying  graph coloring.

In~\cite{kgeq5} we gave a new lower bound on the similar function $\fDP(n,k)$ for so called DP-coloring, defined below (Section~\ref{DPs}). The proof in~\cite{kgeq5} worked for $k\geq 5$, but the ideas were insufficient for $k=4$.
The goal of this paper is to prove a lower bound for $\fDP(n,4)$.
Below, we discuss mostly the state of art for the function $f(n,4)$, while the more general function $f(n,k)$ is discussed  in~\cite{kgeq5}.

In his fundamental papers~\cite{1963Gallai1} and~\cite{1963Gallai2} from 1963, Gallai proved a series of important facts about color-critical graphs. Recall that
a {\em Gallai tree} is a connected graph whose every block is a complete graph or an odd cycle. A~\emph{Gallai forest} is a graph where every connected component is a Gallai tree.

\begin{thm}[Gallai~\cite{1963Gallai2}]\label{Ga2}
Let $k\geq 4$, and let $G$ be a $k$-critical graph. If $B$ is the set of vertices of degree $k-1$ in $G$,
then $G[B]$ is a Gallai forest.
\end{thm}

Gallai derived from
Theorem~\ref{Ga2}  a  bound on $f(n,k)$. For 
$k=4$ the lower bound is as follows:

If  $n\geq 6$, then
\begin{equation}\label{in3}
f(n,4)\geq \left(3+\frac{1}{13}\right)\frac{n}{2}.
\end{equation}

Krivelevich~\cite{1997Krivelevich,1998Krivelevich} improved the  bound on $f(n, 4)$ in~\eqref{in3}  to
\begin{equation}\label{in5}
f(n,4)\geq \left(3+\frac{1}{7}\right)\frac{n}{2}.
\end{equation}
 Then, Kostochka and Yancey~\cite{2014KoYa} proved an asymptotically exact bound:

If   $n\geq 6$, then
\begin{equation}\label{KYa}
    f(n,4)\geq   \left(3+\frac{1}{3}\right)\frac{n}{2}-
\frac{2}{3}.
\end{equation}

 \subsection{Known results on list coloring}
For a set $X$, let $\powerset{X}$ denote the power set of a set $X$, and denote $\bigcup_{v\in X}f(v)$ by $f(X)$.
%A graph is {\em simple} if it has no parallel edges.
Given a vertex $v\in V(G)$, the {\em degree} of $v$, denoted $d_G(v)$, is the number of edges incident with $v$. 
The {\em neighborhood} $N_G(v)$ of a vertex $v$ is the set of vertices adjacent to $v$.
Note that $|N_G(v)| \leq d_G(v)$, and equality holds if and only if $v$ has no incident parallel edges.
% A vertex of degree $d$ (at least $d$) is a {\em $d$-vertex} ({\em $d^+$-vertex}).
For vertex subsets $S_1$ and $S_2$, let $E_G(S_1, S_2)$ denote the set of edges $xy \in E(G)$ where $x\in S_1$ and $y\in S_2$.

 A \emph{list assignment} for a graph $G$ is a function $L \colon V(G) \to \powerset{Y}$, where $Y$ is a set whose elements are referred to as \emph{colors}. 
For each $u \in V(G)$, the set $L(u)$ is called the \emph{list} of $u$ and its elements are said to be \emph{available} for~$u$. A proper coloring $f \colon V(G) \to Y$ is  an \emph{$L$-coloring} if  $f(u) \in L(u)$ for each $u \in V(G)$. A graph~$G$ with a list assignment $L$ is said to be \emph{$L$-colorable} if it admits an $L$-coloring. 
The \emph{list chromatic number} $\chi_\ell(G)$ of $G$  is the least positive integer $k$ such that $G$ is $L$-colorable whenever $L$ is a list assignment for $G$ with $|L(u)| \geq k$ for all $u \in V(G)$. 
%If $L(u) = \{1, \ldots, k\}$ for a positive integer $k$ for all $u \in V(G)$, then $G$ is $L$-colorable if and only if it is $k$-colorable; in this sense, l
The proper coloring problem for $G$ can be expressed as a list coloring problem for $G$ by letting $L(v)$ be the same for each $v \in V(G)$. In particular, $\chi_\ell(G) \geq \chi(G)$ for all graphs $G$.

The definition of a critical graph can be naturally extended to the list coloring setting. 
% A graph $G$ is said to be \emph{$L$-critical}, where $L$ is a list assignment for $G$, if $G$ is not $L$-colorable but for any $u \in V(G)$, the graph $G-u$ is $L\vert_{V(G) \setminus \set{u}}$-colorable. Furthermore,
A graph $G$ is {\em list $k$-critical} if $\chi_\ell(G) \geq k$, but
$\chi_\ell(G') \leq k-1$ for each proper subgraph $G'$ of $G$.
We define $f_\ell(n,k)$  to be the minimum number of edges in an $n$-vertex list $k$-critical graph.

A list assignment $L$ for a graph $G$ is called a \emph{degree list assignment} if $|L(u)| \geq d_G(u)$ for all $u \in V(G)$. The following fundamental result of Borodin~\cite{1979Borodin} and Erd\H os, Rubin, and Taylor~\cite{1980ErRuTa}  can be viewed as a generalization of Theorem~\ref{Ga2}.
	
\begin{thm}[Borodin~\cite{1979Borodin}; Erd\H os, Rubin, and Taylor~\cite{1980ErRuTa}]\label{theo:list_Brooks}
Let $G$ be a connected graph and $L$ be a degree list assignment for $G$. If $G$ is not $L$-colorable, then $G$ is a Gallai tree; furthermore, if $|L(u)| = d_G(u)$ for all $u \in V(G)$ and $u$, $v \in V(G)$ are two adjacent non-cut vertices, then $L(u) = L(v)$.
	\end{thm}

Exactly as Theorem~\ref{Ga2} implies~\eqref{in3}, Theorem~\ref{theo:list_Brooks} yields the same lower bound on~$f_\ell(n,k)$. For $k=4$, the only improvement of it was made by Rabern~\cite{2016Rabern}:
For  $n\geq 6$, 
\begin{equation}
  f_\ell(n,4)\geq
   \left(3+\frac{1}{10}\right)\frac{n}{2}.
\end{equation}

\subsection{ DP-coloring and our results}\label{DPs}

In this paper we study a generalization of list coloring  recently introduced by Dvo\v r\' ak and Postle~\cite{2018DvPo}; 
they called it \emph{correspondence   coloring}, and we call it \emph{DP-coloring} for short. 
%Dvo\v r\' ak and Postle invented DP-coloring in order to approach an open problem about list colorings of planar graphs with no cycles of certain lengths.
%This hints the usefulness of DP-coloring for graph coloring problems, in particular list coloring problems. 

For a multigraph $G$, a \emph{(DP-)cover} of $G$ is a pair $(H, L)$, where $H$ is a graph and $L \colon V(G) \to \powerset{V(H)}$ is a function such that 
\begin{itemize}
    \item The family $\{L(u) : u \in V(G)\}$ forms a partition of $V(H)$. 
    \item For each $u \in V(G)$, $H[L(u)]$ is an independent set.
%    \item If $u$, $v \in V(G)$ and $L(v) \cap N_H(L(u)) \neq \emptyset$, then $v \in \set{u} \cup N_G(u)$. %either $u = v$, or else, $uv \in E(G)$;  
    \item For each $u,v \in V(G)$, if $|E_G(u,v)|=s$, then $E_H(L(u), L(v))$ is the union of $s$ matchings (where each matching is not necessarily perfect and possibly empty). 
\end{itemize}
% {\PB The third item was implied by the fourth one so I commented out the third one.}
We often refer to the vertices of $H$ as \emph{colors}.
A multigraph $G$ with a cover $(H,L)$ has an \emph{$(H,L)$-coloring} if $H$ has an independent set containing exactly one vertex from $L(v)$ for each $v \in V(G)$.
%Equivalently, an independent set $I$ of $H$ is an $(H,L)$-coloring of $G$ if $I \cap L(u) \neq \emptyset$ for all $u \in V(G)$.
The \emph{DP chromatic number} $\cDP(G)$ of a multigraph $G$  is the least positive integer $k$ such that $G$ has an $(H,L)$-coloring
whenever $(H,L)$ is a DP-cover of $G$ with $|L(u)| \geq k$ for all $u \in V(G)$.
Every list coloring problem can be represented as a DP-coloring problem.
In particular, $\cDP(G) \geq \chi_\ell(G)$ for all multigraphs $G$.
%Let $G$ be a graph and $(H,L)$ be a cover of $G$. 

We say that a multigraph $G$ is \emph{DP degree-colorable} if $G$ has an $(H, L)$-coloring whenever $(H, L)$ is a cover of $G$ with $|L(u)|\geq d_G(u)$ for all $u\in V(G)$. 
A multigraph $G$ is \emph{DP $k$-critical} if $\cDP(G) = k$ and $\cDP(G') \leq k-1$ for every proper subgraph $G'$ of $G$.
We let $\fDP(n,k)$ denote the minimum number of edges in an $n$-vertex DP $k$-critical simple graph.

For a simple graph $G$ and a positive integer $t$, the {\em multiple} $G^t$ of $G$ is the multigraph obtained from $G$ by replacing each edge  $uv \in E(G)$
with $t$  edges joining  $u$ and $v$. 
In particular, $G^1=G$.
A \emph{GDP-forest} is a multigraph in which every block is isomorphic to either $K_n^t$ or $C_n^t$ for some $n$ and $t$. 
% (A {\em double cycle} is a graph $C^2_n$.)
A {\em GDP-tree} is a connected GDP-forest. 
Note that every Gallai tree is also a GDP-tree.

Theorems~\ref{Ga2} and~\ref{theo:list_Brooks} extend  to DP-coloring as follows. 

\begin{thm}[Dvo\v{r}\'ak and Postle~\cite{2018DvPo} for simple graphs; Bernshteyn, Kostochka, and Pron’~\cite{2017BeKoPr} for multigraphs]
\label{thm:deg-choosable}
    Suppose that $G$ is a connected multigraph. Then, $G$ is not DP degree-colorable if and only if $G$ is a GDP-tree.
\end{thm}

Analogous to how Theorems~\ref{Ga2} and~\ref{theo:list_Brooks} imply lower bounds on $f(n, k)$ and $f_\ell(n, k)$, respectively, Theorem~\ref{thm:deg-choosable} yields that~\eqref{in3} holds also for $\fDP(n,4)$ when $n\geq 6$. 
%Note that the bound~\eqref{in3} does not hold for DP-coloring of general multigraphs: for example,  cycles
%$C^k_n$ are $2k$-regular and DP $(2k+1)$-critical.

%\begin{thm}\label{thm:DPgallai}
%If $G$ is an $n$-vertex DP $k$-critical multigraph, then $|E(G)|\geq\left(k-1\right)\frac{n}{2}$.
%\end{thm}

The main result of this paper is a lower bound on $\fDP(n,4)$
that is asymptotically better than the one in~\eqref{in3}. Our result implies that 
\begin{equation}\label{ourr}
\mbox{for $n\geq 6$, $n\notin \{7,10\}$, }  \qquad\fDP(n,4)>
   \left(3+\frac{1}{5}\right)\frac{n}{2}.
\end{equation}  
To state the full result, we need to introduce some notions.

For a graph $G$ and a vertex $u$ of $G$, a {\it split of $u$ into $u_1, u_2$} is a new graph $G'$ obtained
 by partitioning $N_G(u)$ into two nonempty subsets $U_1$ and $U_2$, deleting $u$ from $G$, and adding two new vertices $u_1$ and $u_2$, so that $N_{G'}(u_i) = U_i$.
A \emph{DHGO-composition} (named after Dirac, Haj\'os, Gallai, and Ore)
of two graphs $G_1$ and $G_2$ is a new graph obtained by deleting some edge $xy$ from $G_1$, splitting some vertex $z$ from $G_2$ into $z_1, z_2$, and identifying $x, y$ with $z_1, z_2$, respectively. 
%We write $\DHGO(G_1,G_2,x,y,z)$ for the DHGO-decomposition described above.

We recursively define $4$-Ore graphs as follows.
A graph $F$ is a \emph{4-Ore graph} if either $F$ is isomorphic to $K_4$, or $F$ is a DHGO-composition of two $4$-Ore graphs. 
%It is easy to see through induction that
 By definition, each $4$-Ore graph has $3s+1$ vertices and $5s+1$ edges for some integer $s \geq 1$.
 In these terms, our main result is:

\begin{thm}\label{thethm}
    If $G$ is a DP 4-critical graph, then one of the following holds:
    \begin{enumerate}
        \item $G$ is a $4$-Ore graph on at most $10$ vertices.
        %(potential is $3-s$ where graph has $3s+1$ vertices $s\in\{1,2,3\}$)

        \item $|E(G)| \geq \frac{8|V(G)| + 1}{5}$. 
    \end{enumerate}
\end{thm}
Clearly, this result yields~\eqref{ourr}.

We do not know whether $f_{\ell}(n,4)\geq \fDP(n,4)$ for all $n\geq 11$.
However, our method for proving Theorem~\ref{thethm} also allows us to show the following bound on $f_{\ell}(n,4)$.

\begin{cor}\label{cor62}
    For $n\geq 6$,  $n\notin \{7,10\}$,
\begin{equation}\label{ourr2}
  f_{\ell}(n,4)>
   \left(3+\frac{1}{5}\right)\frac{n}{2}.
\end{equation} 
\end{cor}

The structure of the paper is as follows.
In Section~\ref{sec:setup}, we introduce a stronger version of Theorem~\ref{thethm}, which allows multiple edges and variable list sizes of the vertices.
Properties of a minimum counterexample are shown in Section~\ref{sec:propG}, and we finish the proof using discharging
%via the discharging method
in Section~\ref{sec:discharging}.
In Section~\ref{sec:list}, we prove Corollary~\ref{cor62}.

\section{Setup}\label{sec:setup}

In order to prove Theorem~\ref{thethm},
we need a stronger statement.  Our approach is similar to the one in~\cite{kgeq5}, but we need new tricks.
One of the reasons is the need to handle different exceptional graphs, which are in this case $4$-Ore graphs with few vertices.
Let $G$ be a loopless multigraph, and let $h:V(G) \to \mathbb \{0, 1,2,3\}$ be a function. 
For each $v \in V(G)$ and $xy\in {V(G)\choose 2}$, we define the \emph{potential} $\rho_{G,h}(v)$ and $\rho_{G,h}(xy)$ as follows:
\[
\hfill
\rho_{G,h}(v) = 
\begin{cases}
    8 & \textrm{ if } h(v) = 3, \\ 
    4 & \textrm{ if } h(v) = 2, \\ 
    1 & \textrm{ if } h(v) = 1, \\
    -1 & \textrm{ if } h(v) = 0.
 % very tempting to write potential is -1+h(v)*(h(v)+3)/2
\end{cases}
\qquad\hfill\qquad
\rho_{G,h}(xy) = 
\begin{cases}
0 & \textrm { if } |E_G(x,y)| = 0,  \\
1 - 6|E_G(x,y)| & \textrm{ if } |E_G(x,y)| \geq 1.
\end{cases}
\hfill
\]

Given a vertex subset $A \subseteq V(G)$, 
%we
%write $|A|$ for the number of vertices in $A$ and $\|A \|$ for the number of edges in $G[A]$. We 
we define the \emph{potential} of $A$ as 
\[\rho_{G,h}(A) = \sum\nolimits_{v \in V(A)} \rho_{G,h}(v) + \sum\nolimits_{xy \in \binom{A}{2}} \rho_{G,h}(xy) .\]
We also write $\rho_h(G) = \rho_{G,h}(V(G))$. 
We often omit $G$ and $h$ when they are clear from context.
% {\PB $G$ is a general graph, so I don't see a reason to introduce $F$}
% {\PB \sout{When $F$ is a multigraph and}}
 If  $h \equiv 3$ is a constant function on $V(G)$, we often write $\rho_{\mathbf{3}}(G) = \rho_{G,h}(V(G))$.
%Furthermore, when a correspondence cover $(H,L)$ of $G$ is given, we often write $L$ in place of $h$ to represent the function $h:V(G) \rightarrow \{1,2,3\}$ defined as $h(v) = |L(v)|$.
%Note that $\rho(A)$ implicitly depends on the correspondence cover $(H,L)$
%for each $A \subseteq V(G)$.

Given a multigraph $G$, let $h:V(G) \to \N \cup\{0\}$.
A cover $(H,L)$ of $G$ is an {\em $h$-cover} of $G$ if $|L(v)| \geq h(v)$ for each $v \in V(G)$.
When $h\equiv t$ for some constant $t\in \N$, we call an $h$-cover simply a {\em $t$-cover}.
We say that $G$ is {\em DP $h$-colorable} if $G$ has an $(H,L)$-coloring for every $h$-cover $(H,L)$ of $G$. 
We say that $G$ is (DP) {\em $h$-minimal}
if $G$ is not DP $h$-colorable, but every proper subgraph $G'$ of $G$ is DP $h\vert_{V(G')}$-colorable.
Given a vertex subset $A \subseteq V(G)$, we often write $h$ for the restriction $h\vert_{A}$ when this does not cause confusion.

We aim to prove the following strengthening of Theorem~\ref{thethm}.
\begin{thm}
\label{thm:stronger}
    Let $G$ be a loopless multigraph,
    and let $h:V(G) \to \{0, 1,2,3\}$.
     If $G$ is $h$-minimal, then one of the following holds:
    \begin{enumerate}
        \item $G$ is a $4$-Ore graph on at most $10$ vertices, and $h(v) = 3$ for each $v \in V(G)$.
        \item $\rho_h(G) \leq -1$.
    \end{enumerate}
\end{thm}

We say that a pair $(G,h)$ is \emph{exceptional} if $G$ is a $4$-Ore graph on at most $10$ vertices, and $h(v) = 3$ for each $v \in V(G)$.
We note that 
when $G$ 
is a simple graph and $h(v) = 3$ for each $v \in V(G)$,
the statement $\rho_h
(G) \leq -1$ is equivalent to the statement $|E(G)| \geq \frac{8|V(G)| + 1}{5}$. As 
every $4$-Ore graph has $3s+1$ vertices for some integer $s \geq 1$,
Theorem~\ref{thm:stronger} implies Theorem~\ref{thethm}.

\subsection{Preliminaries and $4$-Ore graphs}
% Given a graph $G$ with a DP-cover $(H,L)$ and a vertex $v \in V(G)$, we write $L(v) = \{1_v, \dots, t_v\}$, where $t = |L(v)|$, unless otherwise specified.
% %Given a graph $G$ with an edge ordering, we say that an edge $e \in E(G)$ is \emph{heavy} if $e$ joins $u,v \in V(G)$ and some edge $e'$ appearing before $e$ in the edge ordering also joins $u$ and $v$.
% We write $d_G(v)$ for the number of edges incident with $v$. Note that $|N(v)| \leq d_G(v)$, and equality holds if and only if $v$ is incident with no pair of parallel edges.

%Let $G_0$ be a graph. Given an integer $k \geq 1$, 
%we write $G_0^{(k)}$ for the graph obtained from $G_0$ by replacing each edge with a set of $k$ parallel edges.
%We say that $G_0$ is \emph{DP-degree-colorable} if 
%$G_0$
%is DP-$h$-colorable for the function $h(v) = \deg(v)$.
%The following is a lemma of Bernshteyn, Kostochka, and Pron.

The following lemma is easily verified.
\begin{lemma}
\label{lem:cycle-cover}
Let $C$ be a cycle, and let $(H,L)$ be a $2$-cover of $C$ for which $C$ has no $(H,L)$-coloring.
    \begin{enumerate}
        \item If $|E(C)|$ is odd, then $H$ is isomorphic to two disjoint copies of $C$.
        \item If $|E(C)|$ is even, then $H$ is isomorphic to a single cycle of length $2|E(C)|$.
    \end{enumerate}
\end{lemma}

% \section{Properties of $4$-Ore graphs}
% \label{subsec:Ore}

The unique $4$-Ore graph on $4$ vertices is $K_4$, and the unique Ore graph on $7$ vertices is called the \emph{Moser spindle}.
Every $4$-Ore graph $F$ is $4$-critical, and if $|V(F)| = 3s+1$, then $|E(F)| =
%\frac 13 ( 5(3s+1) - 2) =
5s+1$.

We now establish some useful properties of $4$-Ore graphs.

\begin{obs}\label{obs:ore-potential}
    If $F$ is a $4$-Ore graph on $3s+1$ vertices,  then $\rho_{{\bf 3}}(F) = 3 - s$.
\end{obs}
\begin{proof}
    If $F$ has $3s+1$ vertices, then $F$ has $5s+1$ edges, so $\rho_{{\bf 3}}(F) = 8(3s+1) - 5(5s+1) = 3-s$.
\end{proof}

The following lemma follows directly from~\cite[Claim 16]{2014KoYa} by setting $k = 4$.
Given a multigraph $F$ and a subset $A \subseteq V(F)$, we write $\|A\| = |E(F[A])|$.
\begin{lemma}[Kostochka and Yancey~\cite{2014KoYa}]\label{lem:KY}
For every $4$-Ore graph $F$ and nonempty  $A\subsetneq V(F)$, 
\begin{equation}
\label{eqn:KY}
    5|A| - 3 \|A\| \geq 5.
\end{equation}
\end{lemma}

\begin{lemma}\label{i2}
    If $F$ is a $4$-Ore graph, % and $h(v) = 3$ is constant on $F$, 
    then  every nonempty $A\subsetneq V(F)$ satisfies $\rho_{F,{\mathbf 3}}(A) \geq \rho_{\mathbf 3}(F)+6$.
    % potential of an edge
\end{lemma}
\begin{proof}
    % Let $G$ be a minimum counterexample to the statement where $G$ has $3s+1$ vertices for some positive integer $s$. 
    Let $F$ be a $4$-Ore graph on $3s+1$ vertices.
    By Observation~\ref{obs:ore-potential}, $\rho_{\mathbf 3}(F) = 3-s$,
    so $\rho_{\mathbf 3}(F)+6=9-s$. Hence, we aim to show that $\rho_{F,h}(A) \geq 9-s$.
     % Let $A$ be a nonempty proper subset of $V(G)$ with minimum size such that $\rho(A)\leq (1-k)s+3k+2$.
     We frequently use  the fact that  by~(\ref{eqn:KY}), $\|A\| \leq \frac 13 (5|A| - 5)$.
    
    If $|A|=3s$, then 
    $\|A\| \leq 5s-2$, so
    $\rho_{F,h}(A)\geq 8 (3s)- 5(5s-2)=10-s> 9 - s$.
    If $|A|=3s-1$, then 
    $\|A\| \leq 5s-4$, so
    $\rho_{F,h}(A)\geq 8 (3s-1)-5 (5s-4)=12- s>9-s$.
    Now suppose $|A|\leq 3s-2$, so that
    $-|A|\geq {-(3s-2)}$.
    Then,
    $$\rho_{F,h}(A)=8|A|-5 \|A\|
    \geq 8|A|+ \frac{5(5-5|A|)}{3}
    ={-|A|+25\over 3}
    \geq {-(3s-2) +25 \over 3}
    ={-3s+27 \over 3}
    =9 - s.$$
\end{proof}

Given a multigraph $F$, let $F^-$ denote a multigraph obtained from $F$ by removing any edge. 

\begin{lemma}\label{ore-47}
    Let $F$ be a copy of the Moser spindle, and let $xy\in E(F)$. 
    Suppose $F'$ is formed from $F-xy$ by adding a $3$-cycle $x'y'z'$  and a matching $xx', yy', zz'$ where $z\in V(F)$. If $z$ does not belong to $N(x) \cup N(y)$,
    then $F'$ is either 4-Ore or DP 3-colorable. 
\end{lemma}
\begin{proof}
    As $F$ is isomorphic to the Moser spindle, $F$ is a $4$-Ore graph on seven vertices.
    Fix a 3-cover $(H,L)$ of $F'$.
    \begin{enumerate}[(1)]
        \item If $xy$ is contained in a $3$-cycle $xyw$ in $F$ and $d_F(x) = d_F(y)  =d_F(w)= 3$, then
        as $z \not \in N(x) \cup N(y)$, $z \neq w$.
        Therefore,
        $\Theta:=F'[\{x,y,w,x',y',z'\}]$ is  an induced subgraph         
        isomorphic to a $6$-cycle with one chord,
        and every vertex of $\Theta$ has degree $3$ in $F'$.
        By Theorem~\ref{thm:deg-choosable},
        $F - V(\Theta)$ has an $(H,L)$-coloring $f$, and $f$ can be extended to $\Theta$ by Theorem~\ref{thm:deg-choosable}. 
        Therefore, $F'$ is DP $3$-colorable.

\begin{figure}

\begin{center}
\begin{tikzpicture}
[xscale=2,auto=left, 
blacknode/.style={circle,draw,fill=black,minimum size = 6pt,inner sep=0pt}, 
graynode/.style={circle,draw,fill=gray!30,minimum size = 6pt,inner sep=0pt}
]
% \begin{scope}[xshift=\spacee*0cm]\picpicee\end{scope}
% \begin{scope}[xshift=\spacee*1cm]\picpicf\end{scope}
% \begin{scope}[xshift=\spacee*2cm]\picpicfff\end{scope}
\begin{scope}[xshift=0cm]\picpicd\end{scope} 
%\begin{scope}[xshift=\spacee*3.75cm]\picpica\end{scope}
 %\begin{scope}[xshift=\spacee*4.75cm]\picpicb\end{scope}
 %\begin{scope}[xshift=\spacee*5.75cm]\picpicc \end{scope}
 \begin{scope}[xshift=3cm]\picpicg\end{scope}
\end{tikzpicture}
\end{center}
\caption{
Two $4$-Ore graphs $F'$ that appear in the proof of Lemma~\ref{ore-47}. The light vertices form a Moser spindle minus an edge, and the dark vertices form the $3$-cycle $x'y'z'$ that we add to form $F'$. 
The thick edges show the $K^-_4$ that is composed with a Moser spindle to form the graph.}
\label{fig:4-cases}
\end{figure}
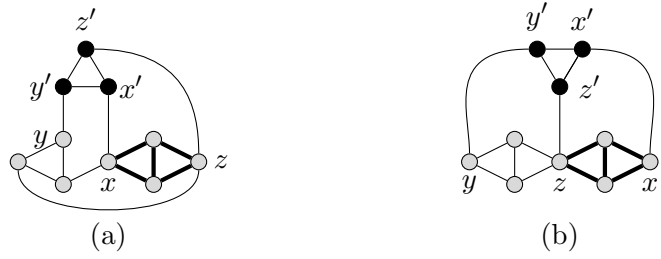

    \item If $x$ is the unique vertex of degree $4$ in $F$, then
    as $z \not \in N(x) \cup N(y)$, 
    $F'$ is isomorphic to Figure~\ref{fig:4-cases}~(a).
    Then,    $F'$     is obtained from a DHGO-composition of $K_4$ and the Moser spindle and hence is a $4$-Ore graph.

    \item 
    If $xy$ is the unique edge in $F$ that does not belong to a $3$-cycle, 
    then as $z \not \in N(x) \cup N(y)$,
    $F'$ is isomorphic to the graph in Figure~\ref{fig:4-cases}~(b). Then,
    $F'$ is obtained from a DHGO-composition of $K_4$ and the Moser spindle and hence is a $4$-Ore graph.
    \end{enumerate}

    This completes the proof.
\end{proof}

\section{Properties of a minimum counterexample}
\label{sec:propG}

For the rest of the paper, we fix a counterexample $G$ to Theorem~\ref{thm:stronger} for which $|V(G)|$ is minimum.
We also fix an $h$-cover $(H,L)$ of $G$ such that $G$ does not have an $(H,L)$-coloring.
For every $uv\in E(G)$, we assume $E_H(L(u), L(v))$ is the union of $|E_G(u, v)|$ maximal matchings.
We write $L(v) = \{1_v, \dots, t_v\}$, where $t = |L(v)|$.
For $A\subseteq V(G)$, we abuse notation and write $(H, L\vert_{A})$ for the cover $(H[L(A)], L\vert_{A})$.

\subsection{General observations}

\begin{obs}
\label{obs:zero}
    $G$ has no vertex $v$ for which $h(v) = 0$. 
\end{obs}
\begin{proof}
    If $V(G) = \{v\}$, then $\rho_h(G) = \rho_h(v) = -1$, so $G$ is not a counterexample. If $|V(G)| \geq 2$, then $G[\{v\}]$ is a proper subgraph of $G$ that is not DP $h$-colorable, contradicting the assumption that $G$ is $h$-minimal.
\end{proof}

\begin{obs}
\label{obs:reducible}
    For each induced subgraph $X \subseteq G$, $X$ is not DP $(h-d_G + d_X)$-colorable.
\end{obs}
\begin{proof}
    Let $(H,L)$ be an $h$-cover of $G$ for which $G$ admits no $(H,L)$-coloring. As $G$ is $h$-minimal, $G - V(X)$ has an $(H,L)$-coloring $f$. 
    Now, 
    for each $x \in V(X)$, let $L'(x) \subseteq L(x)$ consist of the colors in $L(x)$ that are not adjacent to a color of $f$. As $G$ has no $(H,L)$-coloring,  $f$ cannot be extended to $X$; therefore, $X$ has no $(H[L'(V(X))],L')$-coloring. As $|L'(x)| \geq h(x) - d_G(x) + d_X(x) $ for each $x \in V(X)$,  $X$ is not DP $(h-d_G+d_X)$-colorable.
\end{proof}

The following lemma shows that our potential function is submodular;
 for a proof, see e.g.~\cite[Lemma 3.1]{2022JKMX}.

\begin{lemma}
\label{lem:submodularity}
    If $U_1, U_2 \subseteq V(G)$, then
    \[\rho_{G,h}(U_1 \cup U_2) + \rho_{G,h}(U_1 \cap U_2) = \rho_{G,h}(U_1) + \rho_{G,h}(U_2) + \rho_{G,h}(E'),\]
    where $E'=E_G(U_1 \setminus U_2, U_2 \setminus U_1)$. 
    As a corollary, 
    \[\rho_{G,h}(U_1 \cup U_2) + \rho_{G,h}(U_1 \cap U_2) \leq \rho_{G,h}(U_1) + \rho_{G,h}(U_2).\]
\end{lemma}

\begin{obs}
\label{obs:list-degree}
    For each vertex $v \in V(G)$, $h(v) \leq d_G(v)$.
\end{obs}
\begin{proof}
    We let $X = G[\{v\}]$. 
    By Observation~\ref{obs:reducible}, $X$ is not DP $(h-d_G+d_X)$-colorable. This implies that $h(v) - d_G(v) + d_X(v) = h(v) - d(v) \leq 0$, completing the proof.
\end{proof}

We say that a vertex $v \in V(G)$ is \emph{low} if $h(v) = d(v)$.

\begin{lemma}
\label{lem:ERT}
    If $U \subseteq V(G)$ is a set of low vertices, then each block of $G[U]$ is isomorphic to $K_{s}^{t}$ or $C_{s}^{t}$, for some integers $s,t \geq 1$.
\end{lemma}
\begin{proof}
    Suppose that $G[U]$ contains a block $B$ that is isomorphic to neither $K_{s}^t$ nor $C_{s}^{t}$.
    Let $(H,L)$ be an $h$-cover of $G$ for which $G$ has no $(H,L)$-coloring.
    As $G$ is $h$-minimal, $G \setminus V(B)$ has an $(H,L)$-coloring $f$. Now, for each $v \in V(B)$, let $L'(v) \subseteq L(v)$ consist of the colors in $L(v)$ with no neighbor in $f$. Each $v \in V(B)$ is low, so $|L'(v)| \geq d_{B}(v)$.
    Then, $B$ has an $(H,L')$-coloring by Theorem~\ref{thm:deg-choosable}, and thus $G$ is $(H,L)$-colorable, a contradiction.
\end{proof}

{ The following lemma establishes a lower bound on the potentials of subsets of $V(G)$ and is one of our main tools.}
For distinct vertices $u,v$, we say $(u, v)$ is a {\it simple pair} if $|E_G(u, v)|=1$ and is a {\it parallel pair} if $|E_G(u, v)|\geq 2$.
We also say $E_G(S, T)$ has $p$ simple (parallel) pairs if there are $p$ simple (parallel) pairs $(s, t)$ for which $s\in S$ and $t\in T$. 
Given a subset $S \subseteq V(G)$, we write $\overline S = V(G) \setminus S$.

\begin{lemma}
\label{lem:j(k-2)}
Let $i,j \geq { 0}$ and $i+j\geq1$, and let $S \subsetneq V(G)$.
If $E_G(S, \overline{S})$ has $i$ parallel pairs and $j$ simple pairs,
then $\rho_{G,h}(S) \geq 4i+j+1$.
\end{lemma}

\begin{proof}
    Suppose that the lemma is false, and let $i$ and $j$ have minimum $2i+j$ for which the lemma does not hold.
    Then, $G$ has a set $S \subsetneq V(G)$ for which  $\rho_{G,h}(S) \leq 4i+j$ and
    $E_G(S,\overline S)$ 
    has $i$ parallel pairs and $j$ simple pairs.
    We choose $S$ to be a counterexample to the lemma with largest size.

Let $G' = G -S$. For each vertex $v \in V(G')$, let $h'(v) = \max\{0, h(v) - d_G(v) + d_{G'}(v) \} = \max\{ 0,h(v) - |E_G(v,S)|\}$. 
By Observation~\ref{obs:reducible}, $G'$ is not DP $h'$-colorable.
Therefore, there exists $U \subseteq V(G')$ for which $G'[U]$ has a spanning $h'$-minimal subgraph. As $G$ is $h$-minimal and $U \neq V(G)$, it follows that $(G,h)$ and $(G',h')$ do not agree on $U$; therefore, $U$ contains a neighbor $u$ of $S$. As $h'(u) <3 $,  $\rho_{G',h'}(U) \leq -1$.

Now, consider the set $S':=U \cup S$, and write $\ell_1$ and $\ell_2$ for the number of simple pairs and parallel pairs, respectively, in $E_G(S,U)$. 
As $U$ contains a neighbor of $S$, $\ell_1 + \ell_2 \geq 1$.
As $\rho_h(v)-\rho_{h'}(v)\leq  4j' + 7i' $
for each $v \in U$ contained in $j'$ simple pairs and $i'$ parallel pairs of $E_G(v, S)$, 
we have
\begin{equation}
\label{eqn:j(k-2)}
\rho_{G,h}(S') \leq (\rho_{G,h}(U) + 4 \ell_1 + 7\ell_2) + \rho_{G,h}(S) - 5 \ell_1 - 11 \ell_2 \leq -1 - \ell_1 - 4\ell_2 + \rho_{G,h}(S).
\end{equation}

Now, suppose $2i+j = 1$, so that $i=\ell_2=0$ and $j=1$.
Then, $\rho_{G,h}(S) \leq  1$, so as $\ell_1 \geq 1$,
\[\rho_{G,h}(S') \leq -1 - \ell_1 + 1  \leq - 1.\]

Thus, it follows from the maximality of $S$ that $S' = V(G)$. Therefore, $\rho_h(G) \leq - 1$, and $G$ is not a counterexample to Theorem~\ref{thm:stronger}, a contradiction. This completes the case that $2i+j = 1$.

    Next, suppose that $2i+j \geq 2$. Then,
    as $\rho_{G,h}(S) \leq 4i+j$,
    (\ref{eqn:j(k-2)}) implies that 
    \begin{equation}
    \label{eqn:S'}
    \rho_{G,h}(S')     \leq -1 + 4(i-\ell_2) +(j- \ell_1).
    \end{equation}
    If $S' = V(G)$, then
    $U = \overline S$, so $i =\ell_2$ and $j = \ell_1$. Therefore,~\eqref{eqn:S'} implies that $\rho_h(G) \leq -1$, and $G$ is not a counterexample to Theorem~\ref{thm:stronger}. Otherwise, $|E_G(S', \overline{S'})| \geq 1$, so the $2i+j=1$ case implies that $\rho_{G,h}(S') \geq 2$. In both cases by~\eqref{eqn:S'}, 
    either $i > \ell_2$ or $j > \ell_1$,
    so that there is at least one edge in $E_G(S, V(G') \setminus U)$.

    Now, we observe that $E_G(S, V(G') \setminus U)$ has at least $i - \ell_2$ parallel pairs and $j - \ell_1$ simple pairs.
    As $1 \leq 2(i - \ell_2) + (j-\ell_1) < 2i+j$,
     the minimality of $2i+j$
     tells us that $\rho_{G,h}(S') >  4(i - \ell_2) + (j-\ell_1) $, contradicting (\ref{eqn:S'}). 
     %{\AK A nice fact and a nice proof without introducing 1-vertices.}
\end{proof}

\begin{lemma}\label{lem:no1}
    $G$ has no vertex satisfying $h(v) = 1$.
\end{lemma}
\begin{proof}
    By Observations~\ref{obs:zero} and~\ref{obs:reducible}, each $v \in V(G)$ has positive degree, so by Lemma~\ref{lem:j(k-2)}, $\rho_h(v) \geq 2$. 
By the definition of potential, this implies that $h(v) \geq 2$, a contradiction.
\end{proof}

We also obtain the following observation for all vertex subsets of $V(G)$.

\begin{obs}
\label{obs:geq0}
Each nonempty vertex subset $U \subseteq V(G)$ satisfies $\rho_{G,h}(U) \geq 0$. 
\end{obs}
\begin{proof}
If $1 \leq |U| \leq |V(G) | - 1$, then as $G$ is connected, 
$\rho_{G,h}(U) \geq 2$ by Lemma~\ref{lem:j(k-2)}. If $|U| = |V(G)|$, then $\rho_{G,h}(U) = \rho_{G,h}(G) \geq 0$ by our assumption that $G$ is a counterexample to Theorem~\ref{thm:stronger}.
\end{proof}

We say that a subgraph $B \subseteq G$ is \emph{terminal} if $V(B)$ is joined to $\overline{V(B)}$ by a cut edge.
{ The following lemmas establish certain properties of terminal subgraphs of $G$.}

\begin{lemma}
\label{lem:terminal}
    If $G$ has two disjoint terminal subgraphs $B_1, B_2$ for which $\rho(B_1) + \rho(B_2)\geq 5$, then $V(G) = V(B_1) \cup V(B_2)$.
\end{lemma}
\begin{proof}
    Suppose that $G$ has two disjoint terminal subgraphs $B_1$ and $B_2$ where $\rho_{G,h}(B_1)+\rho_{G,h}(B_2)\geq 5$ and $V(B_1)\cup V(B_2)\subsetneq V(G)$. 
    For $i \in \{1,2\}$, let $B_i$ have a cut vertex $x_i$ incident with the single edge of $E_G(B_i, \overline{B_i})$.
    Letting $h'$ agree with $h$ on each vertex of $G$ except that $h'(x_i) = h(x_i) - 1$ for $i \in \{1,2\}$, Observation~\ref{obs:reducible} implies that each $B_i$ is not DP $h'$-colorable.
    Therefore, for each $x_i$, there is a color $c_i \in L(x_i)$ for which every $(H,L\vert_{V(B_{i})})$-coloring $f$ of $ B_{i}$
    assigns $f(x_i) = c_i$.

    Now, consider the graph $G'$ obtained from $B_1 \cup B_2$ by adding the edge $x_1 x_2$. Let $(H',L')$ be a cover of $G'$ obtained from $(H,L)$ as follows: Take the subgraph of $H[L(V(B_1 \cup B_2))]$, and add a single edge joining $c_1$ and $c_2$. Let $L'$ agree with $L$ on $G'$. Then, by our observation above, $G'$ has no $(H',L')$-coloring. Furthermore, 
    \[\rho_{G',h}(G') = \rho_{G,h}(B_1) + \rho_{G,h}(B_2) - \rho_{G,h}(x_1 x_2) \geq 0.\]
    Hence, $G'$ is a counterexample to Theorem~\ref{thm:stronger}  with fewer vertices than $G$, which is a contradiction.
    Therefore, $V(B_1)\cup V(B_2)=V(G)$. 
\end{proof}

\begin{lemma}
    \label{lem:no2in2}
    If $S\subsetneq V(G)$ satisfies $\rho_{G,h}(S) = 2$, 
    then every nonempty proper subset $S_1 \subsetneq S$ satisfies $|E_G(S_1, \overline{S_1})| \geq 2$. In particular, $\rho_{G,h}(S_1) \geq 3$.
\end{lemma}
\begin{proof}
Let $G'=G-S$.
Since $\rho_{G,h}(S)=2$, by Lemma~\ref{lem:j(k-2)}, $|E_G(S, \overline{S})|=1$.
    Suppose that $S_1\subseteq S$ is a proper subset of $S$ satisfying $|E_G(S_1, \overline{S_1})|=1$.
    By Lemma~\ref{lem:j(k-2)}, $\rho_{G,h}(S_1) \geq 2$.
    Write $S_2 = S \setminus S_1$.    
    Let $u_1\in S_1$, $v_1\in V(G')$, and $ u_2, v_2\in S_2$, such that $u_1u_2$ and $v_1v_2$ are the cut edges connecting $S_1$ and $G'$, respectively, with $S_2$ (see Figure~\ref{fig:2in2}).
    As $2 = \rho_{G,h}(S) = \rho_{G,h}(S_1 \cup S_2) = \rho_{G,h}(S_1) + \rho_{G,h}(S_2) - 5$, it follows that $\rho_{G,h}(S_2) \leq 5$.
\begin{figure}
\begin{center}
\begin{tikzpicture}
[scale=1.2,auto=left,every node/.style={circle,fill=gray!30,minimum size = 6pt,inner sep=0pt}]

\draw [draw=gray!40,fill=gray!40,very thick] (0,-0.5) rectangle (2,0.5);
\draw [draw=gray!40,fill=gray!40,very thick] (3,-0.5) rectangle (5,0.5);
\draw [draw=gray!40,fill=gray!40,very thick] (6,-0.5) rectangle (8,0.5);

\node(f3) at (2,0) [draw = black,fill=black] {};
\node(f3) at (3,0) [draw = black,fill=black] {};
\node(f3) at (5,0) [draw = black,fill=black] {};
\node(f3) at (6,0) [draw = black,fill=black] {};

\node(z) at (4,0.75) [draw=white,fill=white,minimum size=0pt,inner sep=0pt] {$5$};
\node(z) at (1,0.75) [draw=white,fill=white,minimum size=0pt,inner sep=0pt] {$2$};
%\node(z) at (7,0.75) [draw=white,fill=white,minimum size=0pt,inner sep=0pt] {$3$};
\node(z) at (2.5,0.25) [draw=white,fill=white,minimum size=0pt,inner sep=0pt] {$-5$};
%\node(z) at (5.5,0.25) [draw=white,fill=white,minimum size=0pt,inner sep=0pt] {$-5$};
\node(z) at (4,0) [draw=gray!40,fill=gray!40,minimum size=0pt,inner sep=0pt] {$S_2$};
\node(z) at (7,0) [draw=gray!40,fill=gray!40,minimum size=0pt,inner sep=0pt] {$G'$};
\node(z) at (1,0) [draw=gray!40,fill=gray!40,minimum size=0pt,inner sep=0pt] {$S_1$};
\node(z) at (4.8,-0.3) [draw=gray!40,fill=gray!40,minimum size=0pt,inner sep=0pt] {$v_2$};
\node(z) at (6.2,-0.3) [draw=gray!40,fill=gray!40,minimum size=0pt,inner sep=0pt] {$v_1$};
\node(z) at (3.2,-0.3) [draw=gray!40,fill=gray!40,minimum size=0pt,inner sep=0pt] {$u_2$};
\node(z) at (1.8,-0.3) [draw=gray!40,fill=gray!40,minimum size=0pt,inner sep=0pt] {$u_1$};

\draw [-] (2,0) to  (3,0) {};
\draw [-] (5,0) to  (6,0) {};

\end{tikzpicture}
\end{center}
\caption{The structure of $G$ described in the proof of Lemma~\ref{lem:no2in2}. 
The numbers indicate $\rho_{G,h}(S_1) = 2$ and  $\rho_{G,h}(S_2) = 5$, so $\rho_{G,h}(S) = \rho_{G,h}(S_1 \cup S_2) = 2$.
Each cut edge has potential $-5$.}
\label{fig:2in2}
\end{figure}
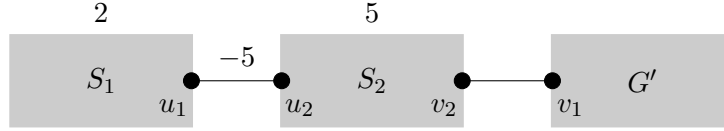

    Every proper subgraph $F$ of $G$ has an $(H,L\vert_{V(F)})$-coloring $f_F$. 
    By Observation~\ref{obs:reducible}, if we define $h'$ on $V(G')$ to agree with $h$ except that $h'(v_1) = h(v_1) - 1$, then $G'$ is not DP $h'$-colorable. 
    Therefore, there exists a color $c_v \in L(v_1)$ such that every possible coloring $f_{G'}$
    assigns $f_{G'}(v_1) = c_v$. 
    Similarly, there exists a color $c_u \in L(u_1)$ such that every possible coloring $f_{G[S_1]}$ assigns $f_{G[S_1]}(u_1) = c_u$.

    Now, we define $G'' = G' \cup S_1 +u_1v_1$ and a cover graph $H''$ from  $H[L(V(G') \cup  S_1 )]$ by adding the edge $c_u c_v$. We observe that $G''$ has no $(H'',L\vert_{V(G'')})$-coloring. 
    Hence, as $G$ is a minimum counterexample, there exists a subset $S'' \subseteq V(G'')$ for which $G''[S'']$ has a spanning $h$-minimal subgraph. 
    As $G$ does not have a proper $h$-minimal subgraph, 
    $S''$ contains  $u_1$ and $v_1$; thus, $G''[S'']$ contains a cut edge and so is not a $4$-Ore graph; in particular, 
    $\rho_{G'',h}(S'')   \leq -1$.     
    Then, in $G$, 
    $$
    \rho_{G,h}(S''\cup S_2)\leq \rho_{G'',h}(S'') - \rho_{G'',h}(u_1v_1) + \rho_{G,h}(S_2) +\rho_{G,h}(u_1u_2) + \rho_{G,h}(v_1v_2) \leq -6+ \rho_{G,h}(S_2) \leq -1,
    $$
    contradicting Observation~\ref{obs:geq0}.

    Finally, $\rho_{G,h}(S_1) \geq 3$ by Lemma~\ref{lem:j(k-2)}.
\end{proof}

\begin{lemma}
\label{lem:2or3}
    Let $S \subseteq V(G)$. If $|E_G(S,\overline S)| = 1$, then $\rho_{G,h}(S) \in \{2,3\}$.
\end{lemma}
\begin{proof}
    By Lemma~\ref{lem:j(k-2)}, $\rho_{G,h}(S) \geq 2$. 
    Now, let $E_G(S, \overline{S})=\{xy\}$ where $x\in S$. 
    Define $h'$ on $S$ so that $h'(v) = h(v)$ for $v \in S \setminus \{x\}$ and $h'(x) = h(x) - 1$. By Observation~\ref{obs:reducible}, $G[S]$ is not DP $h'$-colorable. Therefore, $S$ has a subset $S_1$ for which $G[S_1]$ has a spanning subgraph that is $h'$-minimal. 
    As $S_1$ is a proper subset of $V(G)$, $h$ and $h'$ do not agree on $S_1$; therefore, $S_1$ contains $x$. 
    As $h'(x) < 3$, $\rho_{G,h'}(S_1) \leq -1$, so $\rho_{G,h}(S_1) \leq 3$.
    By symmetry, $\overline S$ has a vertex subset $S_2$ for which $\rho_{G,h}(S_2) \leq 3$.

    Now, $\rho_{G,h}(S_1 \cup S_2) \leq 3 + 3 - 5 = 1$. Thus, by Lemma~\ref{lem:j(k-2)}, $V(G) = S_1 \cup S_2$. Therefore, $S_1 = S$ and $S_2 = \overline S$. So, $\rho_{G,h}(S) \leq 3$, as claimed.
\end{proof}

Recall that $F^-$ denotes a multigraph obtained from a multigraph $F$ by removing any edge.

\begin{lemma}\label{dia}
  $G$ has no $K_4^-$-subgraph
  in which all vertices of a $3$-cycle are low  vertices of degree $3$. 
\end{lemma}
\begin{proof}  
Suppose that $G$ contains a vertex subset $M=\{v_1,v_2,v_3,v_4\}\subset V(G)$ such that
$d(v_1)=d(v_2)=d(v_3)=3$, and such that each vertex pair in $M$ except $v_1, v_4$ is joined by an edge.
 Let $v'_1$ be the unique neighbor of $v_1$ in $V(G) \setminus M$.

By Lemma~\ref{lem:cycle-cover}, $H[L(\{v_1,v_2,v_3\})]$ consists of three $C_3$ components.
As $G$ is $h$-minimal, $G - (M - v_4)$ has an $(H,L\vert_{V(G)\setminus(M-v_4)})$-coloring $f$.
If $f(v_4)$ has neighbors in distinct components of $H[ L(\{v_2,v_3\})]$, then we can extend $f$ to $M$. Therefore, we assume that
for each $i \in \{1,2,3\}$ and pair $x,y \in \{v_1, v_2, v_3\}$, $H$ contains the edge $i_x i_y$.
Moreover, for each $i\in \{1, 2, 3\}$, we may assume $i_{v_4}i_{v_2}, i_{v_4}i_{v_3}\in E(H)$. 

Obtain $G'$ from  $G - \{v_2, v_3\}$ by identifying $v_1$ and $v_4$ into a single vertex $v'$.
Note that as $d_G(v_1) = 3$, $v_1$ and $v_4$ have at most one common neighbor in $V(G) \setminus M$, and hence $v'$ is in at most one parallel pair  in $G'$.
We obtain a cover $(H',L')$ of $G'$ as follows. First, we delete $L(v_2) \cup L(v_3)$ from $H$. Then, for each $i \in \{1,2,3\}$, we identify $i_{v_1}$ and $i_{v_4}$.
Finally, we let $L'(v) = L(v)$ for each $v \in V(G) \setminus \{v_1, v_2, v_3, v_4\}$, and
we let $L'(v')$ consist of the identified colors of $L(v_1)$ and $L(v_4)$. This gives us a cover $(H',L')$ of $G'$.

Now, we claim that if $G'$ has an $(H',L')$-coloring, then $G$ has an $(H,L)$-coloring. Indeed, given an $(H',L')$-coloring $f'$ of $G'$, we can construct an $(H,L)$-coloring of $G$ as follows. For each $v \in V(G) \setminus \{v_1, v_2, v_3, v_4\}$, let $f(v) = f'(v)$. Next, let $i \in \{1,2,3\}$ satisfy $f'(v')  = i_{v'}$, and assign $f(v_1) = i_{v_1}$ and $f(v_4) = i_{v_4}$. 
By our previous observations, $f(v_1)$ and $f(v_4)$ have a common neighbor in $L(v_2)$, so after this step, $L(v_2)$ has two available colors. Finally, let $f(v_3)$ be any available color in $L(v_3)$, and then let $f(v_2)$ be any available color in $L(v_2) \setminus \{f(v_3)\}$. Hence we construct an $(H,L)$-coloring $f$ of $G$.

As we have assumed that $G$ has no $(H,L)$-coloring, it follows that $G'$ has no $(H',L')$-coloring. Therefore, letting $h'(v) = |L'(v)|$ for each $v \in V(G')$, 
$G'$ contains a vertex subset $U$ for which $G'[U]$ is $h'$-minimal. 
As $G'[U]$ is not a subgraph of $G$ by the minimality of the counterexample $G$, it follows that $ v'\in U$.

Now, again by the minimality of  
 $G$, either $\rho_{G',h'}(U) \leq -1$, or $G'[U]$ is a $4$-Ore graph on at most $10$ vertices and $|L'(v)| = 3$ for each $v \in V(G')$. If $\rho_{G',h'}(U) \leq -1$, then as $v'$ is in at most one parallel pair, 
\[\rho_{G,h} ((U - v')\cup M) \leq -1 + 3(8) + 6 - 6(5) = -1,\]
contradicting Observation~\ref{obs:geq0}.

On the other hand, if $G'[U]$ is a $4$-Ore graph, then by definition, $G[(U - v') \cup M]$ is also a $4$-Ore graph on at most $13$ vertices, and $h(v) = 3$ for each $v \in (U - v') \cup M$.
As every $4$-Ore graph is 
{ not DP $3$-colorable},  $V(G) = (U - v') \cup M$. If $|V(G)| = 13$, then $\rho_{G,h}(G) = -1$ by Observation~\ref{obs:ore-potential}, contradicting Observation~\ref{obs:geq0}. If $|V(G)| \leq 10$, then $G$ is a $4$-Ore graph on at most $10$ vertices, and hence $G$ is not a counterexample to Theorem~\ref{thm:stronger}.
\end{proof}

\begin{lemma}\label{lem:no2}
If $v \in V(G)$ satisfies $h(v) = 2$, then $d(v) = 3$ and $v$ has three distinct neighbors.

\end{lemma}
\begin{proof}
   Suppose that $h(v) = 2$. 
   Then, $\rho_h(v) = 4$, so Lemma~\ref{lem:j(k-2)} implies that $v$ is not in any parallel  pair. 
    By Lemma~\ref{lem:j(k-2)}, $d_G(v) \leq 3$.
   For each $w \in N(v)$, Lemma~\ref{lem:no1} implies that $h(w) \in \{2,3\}$.

We write $G' = G - v$.
For each  $w \in N(v)$, we write $h'(w) = h(w) - 1$, and we write
$h'(w)=h(w)$ for $w\not\in N(v)$.
By Observation~\ref{obs:reducible}, $G'$ is not DP $h'$-colorable. Therefore,
 $G'$ has a vertex subset $S$ for which $G'[S]$ has a spanning $h'$-minimal subgraph.
As $G'$ is a proper subgraph of $G$, it follows that $h$ and $h'$ do not agree on $U'$; therefore, $S$ contains some $w \in N(v)$. 
As $h'(w) \leq 2$,
$\rho_{G',h'}(S) \leq -1$.

 Suppose $|N(v) \cap S| = s$, so $1 \leq s\leq 3$. 
 Furthermore,
 \begin{equation}
 \label{eqn:type1}
 \rho_{G,h}(S+v) \leq -1 + 4 + s((8-4) -5) = 3 - s.
\end{equation}
If $s = 3$, then $d_G(v) = 3$, and the proof is complete. 
Otherwise, we consider two cases.

 {\bf Case 1:} Suppose $s = 2$. 
 In this case, (\ref{eqn:type1}) implies that 
 $\rho_{G,h}(S+v) \leq 1$, so Lemma~\ref{lem:j(k-2)} implies that $V(G) = S+v$.
 We write $N(v) = \{u,w\}$.
We assume without loss of generality that $\{1_v 1_u, 2_v 2_u, 1_v 1_w, 2_v 2_w \} \subseteq E(H)$.

Obtain $G_1$ from $G'$ by adding an edge $uw$.
Also, obtain a cover $(H',L)$ of $G_1$ from $H$ by deleting $L(v)$ and then adding the matching $\{1_u2_w,2_u1_w\}$. 
If $G_1$ has an $(H',L)$-coloring $f$, then we can extend $f$ to $v$, because if $f(u)$ is adjacent to the color $i_v \in L(v)$,
then (due to our new edge $uw$) $f(w)$ is not adjacent to $(3-i)_v$.
 Therefore, 
 $G_1$ has no $(H',L)$-coloring and hence has a vertex subset $U$ for which $G_1[U]$ has a spanning $h$-minimal subgraph. 
 As $G_1[U]$ is not a subgraph of $G$, $U$ contains both $u$ and $w$.

As $G$ is a minimum counterexample,
 $\rho_{G_1,h}(U) \leq -1$, or $(G_1[U],h)$ is exceptional.
First, we suppose that $\rho_{G_1,h}(U) \leq -1$.
We observe that
$G[U+v]$ is constructed by deleting an edge $uw$ from $G_1[U]$, adding a vertex $v$ with $h(v) = 2$, and adding two edges $vu$ and $vw$; hence,
 \[\rho_{G,h}(U + v) \leq -1 + 6 - 2(5) + 4 = -1,\]
 contradicting Observation~\ref{obs:geq0}. 
 Therefore, $(G_1[U], h)$ is exceptional. 
We consider several cases.

 \begin{itemize}

\item  If $G_1[U]$ is a $K_4$, then every vertex in $G$ is low. 
Furthermore, $G$ is  not a GDP-tree.
Thus, Theorem~\ref{thm:deg-choosable} implies that $G$ is DP $h$-colorable, a contradiction.

\item If $G_1[U]$ is a $4$-Ore graph on $7$ vertices,
then $G_1[U]$ is isomorphic to the Moser spindle.
Therefore, $\rho_{G_1,h}(U) = 1$ by Observation~\ref{obs:ore-potential}, so $U = V(G_1) = V(G) - v$, and hence $V(G) = U+v$.
We note that $G_1$ has two vertex-disjoint $3$-cycles consisting of vertices of degree $3$, and each of these $3$-cycles belongs to a $K_4^-$-subgraph of $G_1[U]$.
At least one of these $3$-cycles belongs to a $K_4^-$ subgraph of $G$, contradicting Lemma~\ref{dia}.

\item If $G_1[U]$ is a $4$-Ore graph on $10$ vertices, then $|E_{G_1}(u, w)| = 1$; therefore,
\[\rho_G(U+v) \leq \rho_{G_1}(U) + 5 + 4 - 2(5) = 0 - 1 =-1,\]
a contradiction.
\end{itemize}

{\bf Case 2:} Suppose $s = 1$.
In this case, $\rho_{G,h}(S+v) \leq 2$.
By Observation~\ref{obs:list-degree}, $d_G(v) \geq h(v) = 2$. As $1 = s = d_S(v)$, $v$ therefore has a neighbor $w \in V(G) \setminus S$. 
Therefore, $S+v$ is a proper subset of $V(G)$ with potential at most $2$,
and therefore Lemma~\ref{lem:j(k-2)} implies that $E_G(S+v, \overline{S+v}) = \{vw\}$.

Now, we observe that by Lemma~\ref{lem:j(k-2)}, $\rho_{G,h}(S+v) = 2$, implying that $\rho_{G,h}(S) = \rho_{G,h}(S+v) - 4 + 5 = 3$.
Additionally, Lemma~\ref{lem:j(k-2)} implies that $\rho_{G,h}(V(G) \setminus S) \geq 2$, implying that 
$\rho_{G,h}(V(G) \setminus (S+v)) \geq 
\rho_{G,h}(V(G) \setminus S) + 5 - 4 \geq 3$.
Then, $S$ and $V(G) \setminus (S+v)$ are disjoint terminal subgraphs of $G$, both with potential at least $3$ and neither containing $v$, contradicting Lemma~\ref{lem:terminal}.
\end{proof}

\subsection{A special set $S_0^*$ of vertices in $G$}

An {\em edge-block} in a multigraph $G$ is an inclusion maximal connected subgraph $G'$ of $G$ such that
either $|V(G')|=2$ or $G'$ has no cut edges. In particular, every cut edge forms an edge-block, and each connected graph decomposes into edge-blocks.

Define a special subset $S^*_0 \subseteq V(G)$ as follows. If $G$ has no cut edges, then $S^*_0=V(G)$. Otherwise, we fix a smallest pendent edge-block $B^*$ of smallest potential
  and let $S^*_0=V(B^*)$. By Lemma~\ref{lem:no1} and Observation~\ref{obs:list-degree}, a cut edge cannot be a pendent edge-block, so $S^*_0$ is  well-defined.

    If $B^* \subsetneq G$, then
    since $B^*$ is pendent, there are $x^*_0\in S^*_0$ and $y^*_0\in V(G)-S^*_0$ such that $x^*_0y^*_0$ is the unique edge connecting $B^*$ with the rest of $G$. Fix these $B^*$, $S^*_0$, $x^*_0,y^*_0$.
  By definition, $B^*$ is $2$-edge-connected.

The rest of our proof focuses mainly on the set $S_0^*$. 
We first show that no set of low vertices in $S^*_0$ induces a cycle in $G$, and then we use a discharging argument on the vertices of $S_0^*$ to reach a contradiction. One useful property of $S_0^*$ is that since $B^* = G[S_0^*]$ is an edge-block, each proper subset $S \subsetneq S_0^*$ satsifes $|E_G(S, \overline S)| \geq 2$, and hence $\rho_{G,h}(S) \geq 3$ by Lemma~\ref{lem:j(k-2)}. 
The following lemma shows that the same lower bound holds for any subset $S \subseteq V(G)$ where $S\cap S_0^*$ is a nonempty proper subset of $S^*_0$.

\begin{lemma}\label{lem:3-in-2}
If $S \subseteq V(G)$ satisfies $S \cap S_0^* \not \in \{\emptyset,  S_0^*\}$, 
then $\rho_{G,h}(S) \geq 3$. 
\end{lemma}
\begin{proof}
If $G[S]$ is not connected, then it suffices to prove the lemma for a connected component $S'$ of $G[S]$ satisfying $V(S') \cap S_0^* \not \in \{\emptyset, S_0^*\}$.
Therefore, we assume that $G[S]$ is connected.
If $S \subseteq S_0^*$, then as $B^*$ is an edge-block,
$|E_G(S, S_0^* \setminus S)| \geq 2$. Then, by Lemma~\ref{lem:j(k-2)}, $\rho_{G,h}(S) \geq 3$, and we are done. Therefore, we assume that $S_0^* \neq V(G)$ and
$S$ contains $y_0^*$.

    We write $G' = G - S_0^*$. 
    By definition, 
    $S_0^*$ is joined to $G'$ by a cut edge $x_0^*y_0^*$, where $x_0^* \in S_0^*$ and $y_0^* \in V(G')$.

    We first show that $\rho_{h}(G') = 3$. Let $h'$ be defined on $V(G')$ so that $h'(v) = h(v)$ for $v \in V(G')  - y_0^*$, and $h'(y_0^*) = h(y_0^*) - 1$. By Observation~\ref{obs:reducible}, $G'$ is not DP $h'$-colorable. 
    Therefore, $G'$ has some vertex subset $X$ for which $G'[X] = G[X]$ has a spanning $h'$-minimal
    subgraph.
    As $X$ is a proper subset of $V(G)$,
    $h'$ and $h$ do not agree on $X$; therefore,
     $X$ contains $y_0^*$. As $h'(y_0^*) < 3$,
    $\rho_{G,h'}(X) \leq -1$. Thus,
    \[\rho_{G,h}(X) \leq \rho_{G,h'}(X) + 4 \leq 3.\]
    We also note that by Lemma~\ref{lem:2or3},
    $\rho_{G,h}(X \cup S_0^*) = \rho_{G,h}(X) + \rho_{G,h}(S_0^*) - 5 \leq 1$,
    so Lemma~\ref{lem:j(k-2)} implies that $X \cup S_0^* = V(G)$.
    Therefore, $X = V(G')$, and hence $\rho_h(G') \leq 3$.

    Now, by Lemma~\ref{lem:j(k-2)}, $\rho_h(G') \geq 2$. If $\rho_h(G') = 2$, then by Lemma~\ref{lem:no2in2}, $G'$ is an edge-block of $G$. 
    As $B^*$ was chosen to have minimum potential, it follows that $\rho_h(B^*) = 2$. Then, $\rho_h(G) = \rho_h(B^*) + \rho_h(G')-5 = -1$, contradicting Observation~\ref{obs:geq0}. Therefore, $\rho_h(G') = 3$.

    Now, we show that if $A \subsetneq S_0^*$ contains $x_0^*$, then $\rho_{G,h}(A) \geq 5$. 
    As $B^*$ is an edge-block, $|E_G(A,S_0^* \setminus A)| \geq 2$. 
    As $x_0^* y_0^*$ is also incident with $A$, Lemma~\ref{lem:j(k-2)} implies that $\rho_{G,h}(A) \geq 4$.
    Now, if $\rho_{G,h}(A) = 4$, then we write $U' = V(G') \cup A$, and we observe that $\rho_{G,h}(U') = \rho_{h}(G') + \rho_{G,h}(A) - 5 = 2$. 
    Therefore, by Lemma~\ref{lem:j(k-2)}, $U'$ is joined to $V(G) \setminus U'$ by a cut edge, 
    which contradicts the assumption that $B^*$ is an edge-block of $G$.
    Therefore, $\rho_{G,h}(A) \geq 5$.

    Next, we show that if $A \subsetneq V(G')$ contains $y_0^*$, then $\rho_{G,h}(A) \geq 5$. 
     As $A\cup S_0^* \neq V(G)$, $S_0^* \subsetneq A \cup S_0^*$, and $|E_G(S_0^*, \overline{S_0^*})|=1$, Lemma~\ref{lem:no2in2} implies $\rho_{G,h}(A) + \rho_{G,h}(S_0^*) - 5 = 
    \rho_{G,h}(A \cup S_0^*) \geq 3$. 
     Therefore, by Lemma~\ref{lem:2or3}, $\rho_{G,h}(A) \geq 8 - \rho_{G,h}(S_0^*) \geq 5$.

    Finally, consider the set $S$. 
    As $S \not \subseteq S_0^*$,
    $S$ contains both $x_0^*$ and $y_0^*$. 
    As $S \cap S_0^* \neq S_0^*$, 
    \[\rho_{G,h}(S) = -5 + \rho_{G,h}(S \cap V(G')) + \rho_{G,h}(S \cap S_0^*) \geq -5 + 3 + 5 = 3.\]
    This completes the proof.
\end{proof}

\subsection{Subgraphs induced by low vertices}

Recall that a vertex $v \in V(G)$ is \emph{low} if $h(v) = d(v)$.

\begin{lemma}
\label{lem:low-3}
    Every low vertex $v$ in $G$ satisfies $h(v) = d(v) = 3$.
\end{lemma}
\begin{proof}
    By Observation~\ref{obs:zero}, no $v \in V(G)$ satisfies $h(v) = 0$.
    By Lemma~\ref{lem:no1}, no $v \in V(G)$ satisfies $h(v) = 1$.
    If $v \in V(G)$ satisfies $h(v) = 2$, then Lemma~\ref{lem:no2} implies that $d(v) = 3$, and hence $v$ is not low. 
\end{proof}

Let $\mathcal B$ be the subgraph of $G$ induced by low vertices in $G$,
and let ${\mathcal B}_0$ be the subgraph of $\mathcal B$ induced by the vertices of $\mathcal B$ in $S^*_0$.
By Lemma~\ref{lem:ERT}, each block in $\mathcal B$ is a multiple of either a complete graph or a cycle.
As $G$ is $K_4$-free, we hence assume that each block of $\mathcal B$ is a single vertex or a multiple of an edge or a cycle. 
In particular, each such block is regular.
Given an induced cycle $C \subseteq \mathcal B$
 along with distinct vertices $u,u' \in N(C) \setminus V(C)$, we write $F(C,u,u')$ for the multigraph obtained from $G - V(C)$ by adding an edge $uu'$.

\begin{lemma}\label{comp}
If $C$ is an induced cycle of $\mathcal B$, then the multigraph $F(C,u,u')$ is not DP $h$-colorable.
\end{lemma}
\begin{proof}
Let $G'=G-V(C)$.
As $G$ is $h$-minimal, $G'$ has an $(H,L\vert_{V(G')})$-coloring $f$.
 For each $v \in V(C)$, we write $L_f(v)$ for the subset of $L(v)$ consisting of colors with no neighbor in $f$.
We write $H_f = H[ L_f(V(C))]$.
As $|L_f(v)| \geq 2$ for each $v \in V(C)$ and $C$ has no $(H_f,L_f)$-coloring, it follows from Lemma~\ref{lem:cycle-cover} that $(H_f,L_f)$ is $2$-regular. 
By Lemma~\ref{lem:low-3}, $h(v) = 3$ for each $v \in V(C)$;
therefore, $H[L(V(C))] -V(H_f)$ is isomorphic to $C$.
As $f$ is chosen arbitrarily, it follows that for every $(H,L\vert_{V(G')})$-coloring $f$ of $G'$, $f(u)$ and $f(u')$ are adjacent to a component of $H[L(V(C))]$ isomorphic to $C$.

Now, we construct an $h$-cover $(H',L\vert_{V(G')})$ of $F(C,u,u')$ as follows. 
We begin with $(H,L)$, and we remove $L(V(C))$ from $H$. 
Then, we add a matching between $L(u)$ and $L(u')$ so that $c \in L(u)$ and $c' \in L(u')$ are adjacent if and only if $c$ and $c'$ are adjacent to a common component of $H[L(V(C))]$.
This is possible, as $h(v) = 3$ for each $v \in V(C)$, so for each adjacent pair $v,v' \in V(C)$, $L(v)$ and $L(v')$ are joined by a perfect matching.
Then, if $G'$ has an $(H',L\vert_{V(G')})$-coloring $f'$, the colors $f'(u)$ and $f'(u')$ are not both adjacent to a component of $H[L(V(C))]$ isomorphic to $C$, contradicting our observation above.
Therefore, $F(C,u,u')$ is not DP $h$-colorable.
\end{proof}

By Lemma~\ref{comp}, for each induced cycle  $C$ of $\mathcal B$, and any two distinct vertices  $u,u'\in N(C) \setminus V(C)$, 
there is an $h$-minimal multigraph $G(C,u,u')$  contained in $F(C,u,u')$.
As $G$ is $h$-minimal, $G(C,u,u')$ is not a subgraph of $G$; therefore, $G(C,u,u')$ contains $u$ and $u'$.

\subsection{Low cycles in $B^*$}

We aim to show that $\mathcal B_0$ is acyclic.
To this end, we fix a shortest induced cycle $C\subseteq\mathcal B_0$.
 We write $V(C)=\{u_1, \ldots, u_\ell\}$, so that $C$ is an induced cycle in $B^*$ of length $\ell$.
By Lemma~\ref{lem:low-3}, every vertex $u_i \in V(C)$ satisfies $h(v) = d(v) = 3$, so $u_i$ has a unique neighbor $w_i\in V(G) \setminus V(C)$. 
Let $W=\{w_1, \ldots, w_\ell\}$, and note that $|W|\leq \ell$.
Also define $F_{i,j}=G(C,w_i,w_j)$ and $S_{i,j}=V(F_{i,j})$.

\begin{lemma}\label{notx0}
If $S^*_0 \neq V(G)$, then the vertex $x^*_0$ is not on $C$. In particular, $W\subseteq S^*_0$.
\end{lemma}

\begin{proof}
If $S_0^* = V(G)$, then the lemma holds. Hence,
we prove that if $S^*_0 \neq V(G)$, then the vertex $x^*_0$ is not in $C$. To this end,
suppose $x^*_0$ is in $C$.
Since $C$ is $2$-connected, $y^*_0\notin V(C)$. 
If there exists a vertex  $y_1^*\in N(V(C)) \setminus (V(C) \cup \{y^*_0\})$, 
let $G'=G(C,y^*_0,y_1^*)$. Recall that  $y_0^*,y_1^*\in V(G')$. Since $y^*_0y_1^*$ is a cut edge in $G'$, $(G',h)$ is not exceptional.
Therefore,
$\rho_h(G') \leq -1$, and hence $\rho_{G,h}(V(G'))\leq 4$. 
However, the multigraph $G[V(G')]$ is disconnected. 
As $ G[S^*_0]$ is an edge-block of $G$, $|E_G(V(G') \cap S^*_0, S^*_0 \setminus V(G'))| \geq 2$; therefore, by Lemma~\ref{lem:j(k-2)}, $\rho_{G,h}(V(G') \cap S^*_0) \geq 3$. Hence, $\rho_{G,h}(V(G') \setminus S^*_0) \leq 4-3 < 2$, contradicting Lemma~\ref{lem:j(k-2)}.

Suppose now that $V(C)$ has no  neighbor in $V(G) \setminus (V(C) \cup \{y^*_0\})$. 
Then every vertex $u\in V(C)-x_0^*$ satisfies $h(u) = 2$, contradicting Lemma~\ref{lem:low-3}. 
\end{proof}

\begin{lemma}\label{lem:no-low-2cycle}
$C$ is not a $2$-cycle.
\end{lemma}
\begin{proof}
Suppose  the lemma is false, so $u_1$ and $u_2$ are joined by parallel edges. 
If $|E_G(u_1,u_2)| \geq 3$, then $\rho_{G,h}(V(C)) \leq 2(8) - 5 - 6 - 6 = -1$, contradicting Observation~\ref{obs:geq0}. Therefore, $|E_G(u_1,u_2)| = 2$.

Let $G'=G-V(C)$.
As $G$ is $h$-minimal, $G'$ has an $(H,L\vert_{V(G')})$-coloring $f$. 
 As $f$ does not extend to an $(H,L)$-coloring of $G$, 
 it follows that  $G[V(C)]$ does not have an $(H,L\vert_{V(C)})$-coloring using the elements of $L(u_1) \cup L(u_2)$ that are not adjacent to $f$.
In particular, $H[L(u_1)\cup L(u_2)]$ contains a $4$-cycle. 
Thus, we may assume 
$E(H[L(u_1)\cup L(u_2)])=\{1_{u_1}1_{u_2},1_{u_1}2_{u_2},2_{u_1} 1_{u_2},2_{u_1}2_{u_2},3_{u_1}3_{u_2}\}$ and that $H$ contains the edges $f(w_1)3_{u_1},f(w_2)3_{u_2}$.

Consider the cover $(H_1,L_1)$ of $G'$, where $L_1$ differs from $L$ on $G'$ only in that
$L_1(w_1)=L(w_1) \setminus \{f(w_1)\}$,
and $H_1$ is the subgraph of $H$ induced by $L_1(V(G'))$.
Define $h_1:V(G') \to \{0, 1,2,3\}$ so that $h_1(x) = |L_1(x)|$ for each $x \in V(G')$.
If $G'$ has an $(H_1,L_1)$-coloring $f_1$, then we can extend $f_1$ to $G$ by letting $f_1(u_1)=3_{u_1}$ and choosing as $f_1(u_2)$ a color in $\{1_{u_2},2_{u_2}\}$ not adjacent to $f_1(w_2)$. 
Thus $G'$ has no $(H_1,L_1)$-coloring.
Therefore, $G'$ has a vertex subset $W_1$ for which $G'[W_1]$ has a spanning $h_1$-minimal subgraph.
As $G$ is $h$-minimal, 
 $w_1\in W_1$. Since $h_1(w_1) \leq 2$, the minimality of the counterexample $G$ tells us that
$\rho_{G', h_1}(W_1) = \rho_{G,h_1}(W_1)\leq -1$. 
Then 
\[\rho_{G,h}(W_1)\leq \rho_{G,h_1}(W_1) +  \rho_{h}(w_1) - \rho_{h_1}(w_1) \leq -1+4=3.\] 
If $w_2\in W_1$, then
$\rho_{G,h}(W_1\cup  V(C))\leq 3+2(8)-3(5)-(6)=-2,$ contradicting Observation~\ref{obs:geq0}. Thus $w_2\notin W_1$.

By symmetry, there exists a set $W_2 \subseteq V(G')$ satisfying $\rho_{G,h}(W_2) \leq 3$, $w_2 \in W_2$, and $w_1 \not \in W_2$.
Hence, by Lemma~\ref{lem:submodularity}, 
\begin{eqnarray*}
\rho_{G,h}(W_1 \cup W_2) &=& \rho_{G,h}(W_1) + \rho_{G,h}(W_2) - \rho_{G,h}(W_1 \cap W_2) + \rho_{G,h}(E'),  
\end{eqnarray*}
where $E'=E_G(W_1 \setminus W_2, W_2 \setminus W_1)$.
Then,
\begin{eqnarray}
\notag
\rho_{G,h}(W_1 \cup W_2 \cup V(C)) &\leq& \rho_{G,h}(W_1) + \rho_{G,h}(W_2) - \rho_{G,h}(W_1 \cap W_2) + \rho_{G,h}(E') + 2(8) - 3(5)- 6  \\
\label{eqn:WuWv}
&= & \rho_{G,h}(W_1) + \rho_{G,h}(W_2) - 5 - \rho_{G,h}(W_1 \cap W_2) +\rho_{G,h}(E') \\
\notag
&\leq & 1 - \rho_{G,h}(W_1 \cap W_2) + \rho_{G,h}(E')\leq 1.
\end{eqnarray}
As $\rho_G(W_1 \cup W_2 \cup V(C)) < 2$, Lemma~\ref{lem:j(k-2)} implies that $V(G) = W_1 \cup W_2 \cup V(C)$. 
 Furthermore, if $W_1 \cap W_2 \neq \emptyset$, then $\rho_{G,h}(W_1 \cap W_2) \geq 2$ by Lemma~\ref{lem:j(k-2)}, so $\rho_h(G)\leq -1$ by~\eqref{eqn:WuWv}, which is a contradiction since $\rho_h(G)\geq 0$. 
 Thus, 
 $W_1$ and $W_2$ are disjoint, and no edge joins $W_1$ and $W_2$.
Therefore, $G[W_1]$ and $G[W_2]$ are both terminal subgraphs of $G$.

Now, if $\rho_{G,h}(W_1) + \rho_{G,h}(W_2) \leq 4$, then by~\eqref{eqn:WuWv}, $\rho_h(G) \leq -1$, contradicting Observation~\ref{obs:geq0}.
Otherwise, if $\rho_{G,h}(W_1) + \rho_{G,h}(W_2) \geq 5$, then as neither $u_1$ nor $u_2$ belongs to $W_1 \cup W_2$, Lemma~\ref{lem:terminal} is violated. 
Hence, we reach a contradiction, and the proof is complete.
\end{proof}

\begin{lemma}
\label{lem:no-low-triangle}
$C$ is not a $3$-cycle.
\end{lemma}
\begin{proof}
If $w_1 = w_2 = w_3$, then $V(C)\cup W$ induces a $K_4$ in $G$, a contradiction. 
By Lemma~\ref{dia}, we may assume that $w_1, w_2, w_3$ are all distinct.

{\bf Case 1:} $|E(G[W])|\geq 2$.
Without loss of generality, we may assume $w_1w_2, w_2w_3\in E(G)$. 
We write $F=F_{1,3}$ and $S=S_{1,3}$, so 
\[\rho_{G,h}(S \cup V(C) \cup \{w_2\}) \leq 
\begin{cases}
\rho_{G,h}(S) + 3(8) - 6(5) = \rho_{G,h}(S) - 6 & \mbox{ if $w_2\in S$} \\  
 \rho_{G,h}(S) + 4(8) - 8(5) = \rho_{G,h}(S) - 8 & \mbox{ if $w_2\not\in S$}.\\    
\end{cases}\]

Now, if $\rho_h(F) \leq -1$, then $\rho_{G,h}(S) \leq -1 + 6 = 5$, so $\rho_{G,h}(S \cup V(C) \cup \{w_2\} ) \leq 5 - 6 = -1$, contradicting Observation~\ref{obs:geq0}.
Otherwise, $(F,h)$ is exceptional.
Then, $\rho_h(F) \leq \rho_{\mathbf 3}(K_4) = 2$, so $\rho_{G,h}(S) \leq \rho_h(F) + 5 \leq 7$.
If $w_2 \not \in S$, then $\rho_{G,h}(S\cup V(C)\cup\{w_2\})\leq 7-8=-1$, contradicting Observation~\ref{obs:geq0}.

On the other hand, if $w_2\in S$, then by Observation~\ref{obs:geq0},  $0 \leq \rho_{G,h}(S\cup V(C))\leq \rho_{G,h}(S) - 6 \leq 1$.
Thus, by Lemma~\ref{lem:j(k-2)}, $V(G) = S\cup V(C)$, and furthermore,
$\rho_{G,h}(S) \geq 6$.
As $\rho_{G,h}(S) \in \{6,7\}$, it follows that $\rho_{G,h}(S) - 5= \rho_h(F) \in \{1,2\}$,
so by Observation~\ref{obs:ore-potential}, $F$ is a 4-Ore graph $\mathcal O$ on either $7$ or $4$ vertices, and $h \equiv 3$ is a constant function.

Now, write $F' = F_{1,2}$. 
As $|E_{F'}(w_1, w_2)| \geq 2$, $F'$ is not exceptional, so $\rho_h(F') \leq -1$. Therefore, $\rho_{G,h}(V(F')) \leq -1 + 6 = 5$. As $V(F) = V(G) \setminus V(C)$, it follows that $V(F') \subseteq V(F)$.
If $V(F') \subsetneq V(F)$, then by Lemma~\ref{i2}, 
$\rho_{G,h}(V(F')) \geq  \rho_{F,h}(V(F'))  \geq \rho_h(F) + 6 \geq 7$, a contradiction. Therefore, $V(F') = V(F)$, which implies that $\rho_h(G) = \rho_{G,h}(V(F')) + \rho_h(C) - 3(5) \leq 5 + 9 - 15 = -1$, contradicting Observation~\ref{obs:geq0}.

{\bf Case 2:} $|E(G[W])|\leq 1$. 
We may assume $w_1w_2, w_2w_3\notin E(G)$.
For $i\in\{1,3\}$, let $F_i=F_{2,i}$ and write $S_i=V(F_i)$.
Note that $w_1, w_2 \in S_1$, and $w_2, w_3 \in S_3$, and $S_1$ and $S_3$ may not be distinct.
 We note that by Lemma~\ref{notx0}, $w_2 \in S_0^*$.

{\bf Case 2.1:}
$\rho_h(F_1) \leq -1$ and $\rho_h(F_3) \leq -1$. 
In this case,
$\rho_{G,h}(S_1) \leq 4$ and $\rho_{G,h}(S_3) \leq 4$. Since $S_1$ and $S_3$ are both subsets  of $G - V(C)$, $(S_1 \cap S_3) \cap S_0^* \neq S_0^*$. 
Furthermore, as $w_2 \in S_1 \cap S_3 \cap S_0^*$, 
Lemma~\ref{lem:3-in-2} implies that $\rho_{G,h}(S_1 \cap S_3) \geq 3$. 
Then 
$$
\rho_{G,h}(S_1\cup S_3 \cup V(C))\leq \rho_{G,h}(S_1) + \rho_{G,h}(S_3) - \rho_{G,h}(S_1\cap S_3) +\rho_{G,h}(V(C)) - 3(5) \leq 4+4-3-6 = -1,
$$
a contradiction.

{\bf Case 2.2:} 
The pair $(F_1, h)$ is exceptional; that is, $F_1$
is a $4$-Ore graph on which $h\equiv3$ is a constant function.
First, we suppose $(F_3,h)$ is also
exceptional and that $S_1 = S_3$.
In this case, $h$ has a constant value of $3$ on $S_1 \cup V(C)$.
As $w_1, w_2, w_3 \in S_1$, it follows that $\rho_{G,h}(S_1\cup V(C))\leq \rho_{G,h}(S_1) +3(8) - 6(5) = \rho_{G,h}(S_1) - 6$. 
If $|S_1| = 10$, then Observation~\ref{obs:ore-potential} implies that $\rho_{G,h}(S_1 \cup V(C)) \leq \rho_{G,h}(S_1) - 6 \leq -1$, contradicting Observation~\ref{obs:geq0}.
If $|S_1| = 4$, then $F_1$ is isomorphic to $K_4$; however, this contradicts the assumption that $w_1 w_2, w_2w_3 \not \in E(G)$.
Therefore, $|S_1| = 7$.
Furthermore, by Observation~\ref{obs:ore-potential}, $\rho_{G,h}(S_1 \cup V(C)) = (1+5) + 3(8) - 6(5) = 0$. Therefore, by Lemma~\ref{lem:j(k-2)}, $V(G) = S_1 \cup V(C)$.

If $w_1 w_3 \not \in E(G)$, by letting  $(u_1, u_2, u_3)=(x',y',z')$ and applying Lemma~\ref{ore-47}, $G$ is either a $4$-Ore graph or has an $(H,L)$-coloring, a contradiction.
Therefore, $w_1 w_3 \in E(G)$. 
Let $F=F_{1,3}$. 
As $|E_F(w_1,w_3)| \geq 2$,
    $(F,h)$ is not exceptional, so $\rho_h(F) \leq -1$. Thus, $\rho_{G,h}(V(F)) \leq -1 + 6 = 5$.
As $V(F_1)=S_1 = V(G) \setminus V(C)$, it follows that $V(F) \subseteq V(F_1)$. 
If $w_2 \not \in V(F)$, then $V(F) \subsetneq V(F_1)$, so by Lemma~\ref{i2}, $\rho_{G,h}(V(F)) = \rho_{F_1,h}(V(F)) \geq \rho_h(F_1) + 6 = 7$, a contradiction. 
Therefore $w_2 \in V(F)$, so $\rho_h(G) = \rho_{G,h}(V(F) \cup V(C)) \leq 5 + 3(8) - 6(5) = -1$, contradicting Observation~\ref{obs:geq0}. Thus, it does not hold that $(F_3,h)$ is exceptional and that $S_1 = S_3$.

By the arguments above, 
either $\rho_{G,h}(S_3)\leq -1+5=4$, or $(F_3, h)$ is exceptional and $S_1\neq S_3$.

If $S_1 \cap S_3 \neq S_1$, then 
Lemma~\ref{i2} implies that
$\rho_{G,h}(S_1\cap S_3) \geq \rho_{G,h}(S_1) + 6$.
Then,
\begin{eqnarray*}
\rho_{G,h}(S_1\cup S_3 \cup V(C)) &\leq& \rho_{G,h}(S_1) + \rho_{G,h}(S_3) - \rho_{G,h}(S_1\cap S_3) + \rho_{G,h}(V(C)) - 3(5) \\
&\leq& \max\{\rho_{\mathbf 3}(K_4^-), 4\}  - 6-6 = 7 - 12 = -5,
\end{eqnarray*}
contradicting Observation~\ref{obs:geq0}. 
Otherwise, $S_1 \cap S_3 = S_1$. If $(F_3, h)$ is exceptional, then $S_1 \cap S_3 \neq S_3$, and we may apply the previous argument by symmetry. Otherwise, $\rho_{G,h}(S_1) \leq 4$, so
\begin{eqnarray*}
\rho_{G,h}(S_1\cup S_3 \cup V(C)) &\leq& \rho_{G,h}(S_1) + \rho_{G,h}(S_3) - \rho_{G,h}(S_1\cap S_3) + \rho_{G,h}(V(C)) - 3(5) \\
&= & 4+ 3(8) - 6(5) = -2,
\end{eqnarray*}
contradicting Observation~\ref{obs:geq0}.

\end{proof}

Now that we have shown that $C$ is neither a $2$-cycle nor a $3$-cycle, we aim to show that $C$ is also not a cycle of length at least $4$, implying that the low cycle $C$ in fact does not exist.

\begin{lemma}\label{lm1}
    If a subset $A\subseteq V(G)\setminus V(C)$ satisfies both $|N(A)\cap V(C)| \geq j \geq 2$ and $\rho_{G,h}(A)\leq 2j - 1$,  then there exists a subset $A_{\ell}\subseteq V(G)\setminus V(C)$ with $W\subseteq A_{\ell}$ and $\rho_{G,h}(A_{\ell})\leq 2 \ell - 1$ . 
\end{lemma}

\begin{proof}
We obtain $A_{\ell}$ from $A$ by constructing sets $A_i  \subseteq V(G) \setminus V(C)$, for $i \in \{j,\dots,\ell\}$, for which $|N(A_i) \cap V(C)| \geq i$ and $\rho_{G,h}(A_i) \leq 2i-1$.

We begin by letting $A_j = A$.  Now, suppose we have a set $A_i \subseteq V(G) \setminus V(C)$, for which $|N(A) \cap V(C)| \geq i$ and $\rho_{G,h}(A_i) \leq 2i-1$, for some $i \in \{j,\dots,\ell-1\}$. 
We describe how to construct $A_{i+1}$.

If $|N(A_i) \cap V(C)| \geq i+1$, then we let $A_{i+1} = A_i$. Otherwise, suppose $|N(A_i) \cap V(C)| = i$.
We take a vertex, say $w_p$, in $A_i \cap W$ and a vertex,
say $w_q$, in $W \setminus A_i$. 
Let $F=F_{p,q}$ and $S=V(F)$. 
Note that $|N( A_i\cup S)\cap V(C)|\geq i+1$.

We consider two cases. First, suppose that $\rho_h(F) \leq -1$. Then $\rho_{G,h}(S) \leq -1 + 6 = 5$.
Since $S$ contains the vertex $w_p$, which belongs to $S_0^*$ by Lemma~\ref{notx0}, 
it follows that $A_i \cap S \cap S_0^* \not\in\{\emptyset, S_0^*\}$.
Therefore, 
by Lemma~\ref{lem:3-in-2}, $\rho_{G,h}( A_i \cap S)\geq 3$.
Thus $$
\rho_{G,h}(A_i \cup S) \leq \rho_{G,h}(A_i) + \rho_{G,h}(S) - \rho_{G,h}(A_i \cap S) \leq 2i-1 + 5 - 3 = 2(i+1) - 1.
$$
Then, we let $A_{i+1} = A_i \cup S$.

On the other hand, suppose that $F$ is exceptional and that $h(v) = 3$ for each $v \in S$.
By Lemma~\ref{i2}, as $A_i \cap S \neq S$,
\[\rho_{G,h}(A_i \cap S)\geq \rho_{F,h}(S) + 6 = \rho_{G,h}(S) + 1.\]
Therefore, $\rho_{G,h}(A_i \cup S) \leq \rho_{G,h}(A_i) + \rho_{G,h}(S) - \rho_{G,h}(A_i \cap S) \leq \rho_{G,h}(A_i) -1 < 2(i+1)-1$.
Then, we let $A_{i+1} = A \cup S$.

At the end of this procedure, we have a set $A_{\ell} \subseteq V(G) \setminus V(C)$ for which $W\subseteq A_{\ell}$ and $\rho_{G,h}(A_{\ell})\leq 2 \ell - 1$.
\end{proof}

\begin{cor}
\label{cor:W-high-potential}
 If $A \subseteq V(G) \setminus V(C)$ satisifes $|N(A) \cap V(C)| \geq 2$, then $\rho_{G,h}(A) \geq 2|N(A) \cap V(C)|$.
\end{cor}
\begin{proof}
    Write $j = |N(A) \cap V(C)|$, and suppose that $\rho_{G,h}(A) \leq 2j-1$. Then, 
    Lemma~\ref{lm1} implies that there exists a subset $A_0 \subseteq V(G) \setminus V(C)$ satisfying $W \subseteq A_0$ and $\rho_{G,h}(A_0) \leq 2 \ell-1$. Then, 
    \[\rho_{G,h}(A_0 \cup V(C)) \leq \rho_{G,h}(A_0) + \rho_{G,h}(V(C)) - 5 \ell \leq (2\ell-1) + (8\ell-5\ell) - 5\ell = -1,\]
    contradicting Observation~\ref{obs:geq0}.
\end{proof}

\begin{lemma}\label{lem:no-low-cycle}
$B^*$ has no cycle consisting of low vertices.
\end{lemma}
\begin{proof} 
Suppose for the sake of contradiction that a shortest induced cycle $C$  exists in $B^*$.
By Lemmas~\ref{lem:no-low-2cycle} and~\ref{lem:no-low-triangle}, $\ell\geq 4$.

We consider several cases based on the number of vertices in $W$.
If $|W| = 1$, then by Corollary~\ref{cor:W-high-potential}, $8 \geq \rho_{G,h}(W) \geq 2 \ell$, implying that $\ell = 4$, and hence that $\rho_h(W) = 8$. We write $W = \{w\}$ and observe that $h(w) = 3$ by Lemma~\ref{lem:no2}. Also,
$\rho_{G,h}(V(C) \cup \{w\}) = 5(8) - 8(5) = 0 $, and hence Lemma~\ref{lem:j(k-2)} implies that $V(G) = V(C)\cup \{w\}$.
Then, consider a $3$-cover $H$ of $G$. By Lemma~\ref{lem:cycle-cover},
$H[L(V(C))]$ consists of a $C_4$ component and a $C_8$ component. 
Let $c \in L(w)$ be a color that is adjacent to the $C_8$ component in $H[L(V(C))]$, and assign $f(w) = c$.
Then, as the colors of 
$H[L(V(C))]$ not adjacent to $f$
do not induce a $C_8$ subgraph of $H$, $f$
can be extended to an $(H,L)$-coloring of $G$.
Therefore, $G$ is not $h$-minimal, a contradiction.

Hence, we assume that $|W| \geq 2$. 

\begin{claim}
\label{claim:all-W}
 If $w_i\neq w_j$, then
    $|N(S_{i,j}) \cap V(C)| \leq 3$. In particular,    
    $|S_{i,j} \cap W| < |W|$.
\end{claim}
\begin{proof}
    Suppose for the sake of contradiction that 
    $|N(S_{i,j}) \cap V(C)| \geq 4$.
    Then, by Corollary~\ref{cor:W-high-potential}, $\rho_{G,h}(S_{i,j}) \geq 8$.
    However, as $F_{i,j}$ is $h$-minimal,
    either $\rho_{ F_{i,j}}(S_{i,j}) \leq -1$, or ($F_{i,j},h)$ is exceptional.
    If $\rho_{h}(F_{i,j}) \leq -1$,
    then $\rho_{G,h}(S_{i,j}) \leq -1+6=5$. If $(F_{i,j},h)$ is exceptional, then $\rho_{h}(F_{i,j}) \leq 2$ and $w_iw_j \not \in E(G)$; therefore, $\rho_{G,h}(S_{i,j}) \leq 7$.
    In both cases, $\rho_{G,h}(S_{i,j}) \leq 7$, giving us a contradiction.
\end{proof}

Now, we are ready to prove that $B^*$ contains no low cycle of length at least $4$.

If $|W| = 2$,  then,  writing $W = \{w_1,w_2\}$, 
 $S_{1,2}$ contains all of $W$, contradicting Claim~\ref{claim:all-W}.

Therefore,
$|W| \geq 3$. First, suppose that some vertex in $W$, say $w_1$, has at least two neighbors in $C$.
We fix $w_2 \in W \setminus \{w_1\}$.
By Claim~\ref{claim:all-W}, there exists a vertex  $w_3 \in W \setminus S_{1,2}$. 
Furthermore, by Claim~\ref{claim:all-W}, $w_2\notin S_{1,3}$. 
We write $S_2 = S_{1,2}$ and $S_3 =S_{1,3}$.

By Corollary~\ref{cor:W-high-potential}, 
$\rho_{G,h}(S_2)\geq 2|N(S_2)\cap V(C)|\geq 2(3)=6$.
Therefore, $\rho(F_{1,2}) \geq \rho_{G,h}(S_2) - 6 \geq 0$,
so $(F_{1,2} ,h)$ is exceptional.
In particular, $w_1w_2 \not \in E(G)$.
Therefore, $\rho(F_{1,2})  \geq \rho_{G,h}(S_2)-5=1$, so
Observation~\ref{obs:ore-potential} implies that $|S_2| \in \{4,7\}$.
By symmetry, $|S_3| \in \{4,7\}$.
Thus each of $F_{1,2}$ and $F_{1,3}$ is isomorphic to either a $K_4$ or the Moser spindle.

Since $w_1\in S_2 \cap S_3$, and $w_2\in S_2\setminus S_3$, by Lemma~\ref{i2}, \[\rho_{G,h}(S_2\cap S_3)\geq \rho(F_{1,2}) + 6 = \rho_{G,h}(S_2)  + 1.\]
Therefore,
\begin{eqnarray*}
\rho_{G,h}(S_2\cup S_3) &\leq& \rho_{G,h}(S_2) + \rho_{G,h}(S_3) - \rho_{G,h}(S_2\cap S_3) 
\leq \rho_{G,h}(S_3) - 1\leq \rho(K_4) +5 - 1 = 6.
\end{eqnarray*}
This contradicts Corollary~\ref{cor:W-high-potential}, as $|N(S_2\cup S_3)\cap V(C)|\geq 4$.

Therefore, each vertex of 
$W$ has exactly one neighbor in $C$, and hence $|W|=\ell$. 
To finish the proof, we now consider several cases.

{\bf Case 1.} $W$ contains distinct vertices $w_i$ and $w_j$,
for which $w_iw_j \not \in E(G)$ and $\rho_h(F_{i,j}) \leq -1$.
We write $S = S_{i,j}$.
Then, $\rho_{G,h}(S) \leq -1 + 5 = 4$.
By setting $ A = S$ in Corollary~\ref{cor:W-high-potential}, $|S \cap W| = |N(S) \cap V(C)| \leq \frac 12 \rho_{G,h}(S)  \leq 2$. 
As $S$ contains $w_i$ and $w_j$, it follows that $|S \cap W| = |N(S) \cap V(C)| =  2$.
Furthermore, as $V(C) \subseteq S_0^*$ and $S_0^*$ is an edge-block, it follows without loss of generality that $w_i \in S_0^*$.

Suppose some $w_r \in W \setminus S$ 
satisfies $w_iw_r\notin E(G)$.
We write $S' = S_{i,r}$.
If $(F_{i,r},h)$ is exceptional,
then by Lemma~\ref{i2}, 
\[\rho_{G,h}(S\cap S') = \rho_{ F_{i,r} ,h}(S \cap S') \geq \rho_h(F_{i,r}) + 6 = \rho_{G,h}(S') + 1.\]
Therefore, $\rho_{G,h}(S \cup S') \leq \rho_{G,h}(S) + \rho_{G,h}(S') - \rho_{G,h}(S \cap S') \leq \rho_{G,h}(S)- 1 \leq 3$.
However, as $w_i,w_j,w_r \in S \cup S'$,
Corollary~\ref{cor:W-high-potential} implies that $\rho_{G,h}(S \cup S') \geq 6$, a contradiction.

On the other hand, suppose that
$\rho(F_{i,r}) \leq -1$ and $w_iw_r \not \in E(G)$, so that
$\rho_{G,h}(S') \leq -1+5=4$.
As $w_i \in S \cap S'$ and $w_i \in S_0^*$ by Lemma~\ref{notx0},
$S \cap S' \cap S_0^* \not \in \{\emptyset, S_0^*\}$, so $\rho_{G,h}(S \cap S') \geq 3$ by Lemma~\ref{lem:3-in-2}. 
Then,
$$
\rho_{G,h}(S\cup S')\leq \rho_{G,h}(S)+\rho_{G,h}(S') - \rho_{G,h}(S\cap S')\leq 4 + 4-3 = 5.
$$
Again, as $w_i, w_j,w_r \in S \cup S'$, we have $|N(S \cup S') \cap V(C)| \geq 3$, so Corollary~\ref{cor:W-high-potential} implies that $\rho_{G,h}(S \cup S') \geq 6$,
giving us a contradiction again.

By our contradiction, we may assume that for each $w_r \in W \setminus S$, $w_iw_r \in E(G)$. 
By a symmetric argument, for each $w_r \in W \setminus S$, $w_jw_r \in E(G)$.
Then, for each $w_r \in W \setminus S$,
$\rho_{G,h}(S + w_r) \leq \rho_{G,h}(S)+8 - 2(5) \leq 4+8 - 10 = 2$, again contradicting Corollary~\ref{cor:W-high-potential}.

\medskip
{\bf Case 2.}  
%Suppose that 
 Case 1 does not hold.
We write $G' = G -V(C)$.
Choose a set $S \subseteq V(G')$ of smallest potential over all subsets of $V(G')$ containing at least two vertices of $W$.

First, suppose that $|S \cap W| = \ell$, so that $S \cap W = W$. As $G$ is $K_{\ell}$-free, there exist two nonadjacent vertices $w_i, w_j$,  in $W$.
Since Case 1 does not hold, $(F_{i,j},h)$ is exceptional, so
$\rho_h(F_{i,j}) \leq \rho_{\mathbf 3}(K_4) = 2$, thus $\rho_{G,h}(S_{i,j}) \leq 2+5=7$.
As $S$ is chosen to have minimum potential, it follows that $\rho_{G,h}(S) \leq 7$. However, Corollary~\ref{cor:W-high-potential} implies that $\rho_{G,h}(S) \geq 2|N(S) \cap V(C)| = 2\ell \geq 8$, a contradiction.

On the other hand, suppose that $|S \cap W| < \ell$. We fix a vertex $w_i\in W \setminus S$. 
If $w_i$ is adjacent to each vertex in $W \cap S$, then $\rho_{G,h}(S \cup \{w_i\})\leq \rho_{G,h}(S) + 8 - 2(5) < \rho_{G,h}(S)$, contradicting the minimality of $\rho_{G,h}(S)$.

Otherwise, 
 if $W \cap S$ contains a vertex $w_j$ for which $w_iw_j \not \in E(G)$, then we let $S_0 = S_{i,j}$. 
    As we are not in Case 1, $(F_{i,j},h)$ is exceptional. 
    Therefore, Lemma~\ref{i2} implies that $\rho_{G,h}(S_0 \cap S) = \rho_{ F_{i,j},h}(S_0 \cap S) \geq \rho_{ F_{i,j},h}(S_0) + 6 = \rho_{G,h}(S_0) + 1$. Then,
    \[\rho_{G,h}(S \cup S_0 ) \leq \rho_{G,h}(S) + \rho_{G,h}(S_0) - \rho_{G,h}(S_0 \cap S) < \rho_{G,h}(S), \]
    again contradicting the minimality of $\rho_{G,h}(S)$.
\end{proof}

\section{Discharging}\label{sec:discharging}

In this section, we use discharging on the set $S_0^*$ to reach a contradiction with the existence
of the counterexample $G$ to Theorem~\ref{thm:stronger}.  
The rules of discharging are as follows.

\begin{enumerate}[(R1)]
    \item  For each $v \in V(G)$, we give initial charge $\rho_{h}(v)$ to $v$. 
    For each pair $uv$, where $u,v\in V(G)$ are joined by $s \geq 1$ edges, we give initial charge 
    $1-6s$ to the pair $uv$. For each nonadjacent pair $u,v \in V(G)$, the initial charge of $uv$ is $0$.
    \item For each pair $uv$ of adjacent vertices in $G$, if $ s\geq 1$ edges connect $u$ with $v$, 
    the pair $uv$ receives charge
    $(6s-1)/2$ from each of $u$ and $v$. 
    \item Each vertex $u \in S^*_0 \setminus V(\mathcal B_0)$ takes charge $\frac 12$ along each edge that joins $u$ to a vertex $v$ of $\mathcal B_0$.
\end{enumerate}

For each $v \in  S_0^*$, we write $ch^*(v)$ for the final charge of $v$. We also write $ch^*(S_0^*) = \sum_{v \in S_0^*} ch^*(v)$.

After (R1), the charge in $S_0^*$ is $\rho_{G,h}(S_0^*) = \rho_h(B^*)$.
If $S_0^* = V(G)$, then during (R2), $S_0^*$ does not gain or lose any charge. 
If $S_0^* \neq V(G)$, then $S_0^*$ gives charge $\frac 52$ to the edge $x_0^* y_0^*$.
Finally, during (R3), $S_0^*$ does not gain or lose any charge.
Therefore, 
\begin{equation}
\label{eqn:ch-LB}
ch^*(S_0^*) = 
\begin{cases}
\rho_h(G) &\geq 0 \quad \textrm{ if }  S_0^* = V(G), \\
\rho_{G,h}(S_0^*) - \frac 52 &\geq - \frac 12  \quad\textrm{ if } S_0^* \neq V(G).
\end{cases}
\end{equation}

We first show that $\mathcal B_0$ is nonempty. 
Indeed, suppose $\mathcal B_0$ is empty, so $d(v)>h(v)$ for each $v \in S_0^*$. 
Fix $v \in S_0^*$.
If $d(v) \leq 3$, then Lemma~\ref{lem:no1} implies that $h(v) = 2$, so $d(v) = 3$ by Lemma~\ref{lem:no2}. 
Thus, $ch^*(v) = 4 - 3 \left ( \frac 52 \right ) = - \frac 72$.
 If $d(v) \geq 4$,  then $h(v)= 3$, so
  $ch^*(v) \leq 8 - \frac 52 d(v) \leq -2$.
  Therefore, $ch^*(S_0^*) \leq -2 |S_0^*| < -1$, contradicting (\ref{eqn:ch-LB}).
   Hence, $\mathcal B_0$ is nonempty.

Now, consider a vertex $v \in S_0^*$.
If $ d(v) \geq  1+h(v)$, then
we consider several cases.
\begin{enumerate}[(N1)]
    \item \label{item:small-nonlow} 
    If $d(v) \leq 3$, then $h(v) = 2$ by Lemma~\ref{lem:no1}, so $d(v) = 3$ by Lemma~\ref{lem:no2}. Thus,
    \[ch^*(v) \leq 4 - 3 \left (\frac 52 \right )  + \frac 12 (3) = -2.\]
    \item If $d(v) \geq 4$, then
    \[ch^*(v) \leq 8 - \frac 52 d(v)  + \frac 12 d(v)   \leq 0.\]
\end{enumerate}
Furthermore, it is easy to check that 
the inequality in (N2) is strict 
if $d(v) \geq 5$,
$v$ has at most $3$ low neighbors, 
or $v$ is in a parallel pair.

For each low vertex $v \in S_0^*$, Lemma~\ref{lem:low-3} implies that $h(v) = d_G(v) = 3$, so
\[ch^*(v) \leq 8 - 3 \left ( \frac 52 \right ) - \frac 12 (d_{B^*}(v) - d_{\mathcal B_0}(v)) =
\begin{cases}
-1 + \frac 12 d_{ \mathcal B_0 }(v) & \textrm{ if } v \neq x_0^*, \\
- \frac 12 + \frac 12 d_{ \mathcal B_0}(v) & \textrm{ if } v = x_0^*.
\end{cases}
 \]

Thus, for each component $T$ of $\mathcal B_0$, we have 
\[\sum_{v \in V(T)} ch^*(v) \leq 
\begin{cases}
- |V(T)| + |E(T)| & \textrm{ if } x_0^* \not \in V(T), \\
\frac 12 - |V(T)| + |E(T)| & \textrm{ if } x_0^* \in V(T).
\end{cases}
 \]
By Lemma~\ref{lem:no-low-cycle}, $T$ is a tree. Therefore, $|V(T)| = |E(T)| + 1$, so 
\begin{equation}
\label{eqn:ch-low}
\sum_{v \in V( T )} ch^*(v) \leq 
\begin{cases}
-1 \textrm{ if } x_0^* \not \in V(T), \\
-\frac 12 \textrm{ if } x_0^* \in V(T).
\end{cases} 
\end{equation}
Therefore, by (N1), (N2),  (\ref{eqn:ch-low}), and the fact that $\mathcal B_0$ is nonempty,
\[ch^*(S_0^*) = \sum_{v \in V(\mathcal B_0)} ch^*(v) + \sum_{v \in S_0^* \setminus V(\mathcal B_0)} ch^*(v) \leq 
\begin{cases}
-1 \textrm{ if } S_0^* = V(G) \\
-\frac 12 \textrm{ if } S_0^* \neq  V(G),
\end{cases}\]
with equality holding only if each $v \in S_0^* \setminus V(\mathcal B_0)$ satisfies $ch^*(v) = 0$ and $\mathcal B_0$ has exactly one component. 
By (\ref{eqn:ch-LB}), 
we have a contradiction if $S_0^* = V(G)$. Therefore, $S_0^* \neq V(G)$, and $ch^*(S_0^*) = - \frac 12$.
Hence, we may assume that $\mathcal B_0$ has a single component $T$ and that each $v \in S_0^* \setminus V(T)$ satisfies $d(v) = 4$ and has four distinct neighbors in $T$. As $S_0^*$ is an edge-block, it follows that $V(T) \subsetneq S_0^*$, so $|V(T)| \geq 4$.

Now, let $t$ be a leaf of $T$ with at least two neighbors $x,y \in S^*_0 \setminus V(T)$. 
We write $X = G[V(T) \cup \{x,y\}]$, and we claim that $X$ is DP $(3 - d_G + d_X)$-colorable, contradicting Observation~\ref{obs:reducible}.
To this end, for each vertex $v \in V(X)$, write $h'(v) = 3 - d_G(v) + d_X(v)$. We observe that for each $v \in V(T)$, $h'(v) = d_X(v)$, and for each $v \in V(X) \setminus V(T)$, $h'(v) = 3$. 
Let $(H',L')$ be an $h'$-cover of $X$ where $L'=L\vert_{V(X)}$.
We construct an $(H',L')$-coloring $f$ of $X$ as follows.
First,
we select a color $c \in L(t)$, and for each vertex $w \in \{x,y\}$, we assign $f(w)$ to be the neighbor of $c$ in $H'$.
Note that this is possible, as $x$ and $y$ are not adjacent.
Then, we observe that $t$ has $d_T(t) + 1$ available colors in $L'(t)$, and each other vertex $v \in V(T)$ has at least $d_T(v)$ available colors in $L'(v)$. Therefore, we can extend $f$ to an $(H',L')$-coloring of $X$. Hence, $X$ is DP $(3 - d_G + d_X)$-colorable, contradicting Observation~\ref{obs:reducible}. This final contradiction completes the proof.

\section{Bound for list coloring}\label{sec:list}

We say that $G$ is \emph{$(H,L)$-minimal} if $G$ has no $(H,L)$-coloring but every proper subgraph $G'$ of $G$ has an $(H, L\vert_{V(G')})$-coloring.

\begin{thm}
\label{thm:HLcrit}
    Let $G$ be a multigraph
    with a function $h:V(G) \to \{0, 1,2, 3\}$. Suppose that $G$ has no exceptional subgraph. 
    If $(H,L)$ is an $h$-cover of $G$ and $G$ is $(H,L)$-minimal, then $\rho_h(G) \leq -1$.
\end{thm}
\begin{proof}
    Suppose for the sake of contradiction that $\rho_h(G) \geq 0$.
     As $G$ is $(H,L)$-minimal, no vertex $v \in V(G)$ satisfies $h(v) = 0$.
    By Theorem~\ref{thm:stronger}, $G$ is not $h$-minimal.
    Since $G$ is not DP $h$-colorable, $G$ has a proper subgraph $G_0$ that is $h$-minimal. 
    Since $G$ has no exceptional subgraph, Theorem~\ref{thm:stronger} implies that $\rho_{G,h}(V(G_0)) \leq -1$.
    We let $G_1 \supseteq G_0$ be a maximal subgraph of $G$ for which $\rho_{G,h}(V(G_1)) \leq -1$, and we observe that $V(G_1) \subsetneq V(G)$.
    By the $(H,L)$-minimality of $G$, $G_1$ has an $(H,L)$-coloring $f$.

    Now, define $G' = G - V(G_1)$.
    We observe that for each $v \in V(G')$, if $|E_G(v,V(G_1))| = j$, then $h(v) \geq j+1$. 
    Indeed, if $j=0$, then the inequality is clear.
    If $j \geq 1$ and $h(v) \leq j$, then $\rho_h(v) \leq 3j-1$. Therefore,
    $\rho_{G,h}(V(G_1) \cup \{v\}) \leq -1 + (3j-1) - 5j < -1$, contradicting the maximality of $G_1$.
    
    For each vertex $v \in V(G')$, write $h'(v) = h(v) - |E_G(v,V(G_1))|$. By our previous observation, $h'(v) \geq 1$ for each $v \in V(G')$.
    Consider a vertex subset $U \subseteq V(G')$, and let $j = |E_G(U,V(G_1))|$. By the maximality of $G_1$, 
    \[0 \leq \rho_{G,h}(V(G_1) \cup U) \leq -1 - 5j + \rho_{G,h}(U).\]
    Rearranging, $\rho_{G,h}(U) \geq 5j+1 $, which implies that $\rho_{G,h'}(U) \geq 1 + j$.
    Thus,
    every subgraph $G''$ of $G'$ satisfies $\rho_{h'}(G'') \geq 1$, and hence by Theorem~\ref{thm:stronger}, $G'$ has no $h'$-critical subgraph. Therefore, $G'$ is DP $h'$-colorable, and thus 
    $f$ can be extended to an $(H,L)$-coloring of $G$, a contradiction. 
\end{proof}

\begin{cor}\label{cor621}
    If $n \geq 6$ and $n \not \in \{7,10\}$, then $f_{\ell}(n,4) \geq \frac 85 n + \frac 15$.
\end{cor}
\begin{proof}
    Let $G$ be a list $4$-critical multigraph on $n \geq 6$ vertices and $m$ edges, where $n \not \in \{7,10\}$. As $G$ has no $4$-critical proper subgraph, $n \geq 6$ and $n \not \in \{7,10\}$, $G$ has no exceptional subgraph.
    Let $L': V(G) \to \powerset{\N}$ be a list assignment for which $|L'(v)|\geq 3$ for every $v\in V(G)$ and $G$ has no $L'$-coloring. 

    Now, we construct a $3$-cover $(H,L)$ of $G$ for which the $L'$-coloring and the $(H,L)$-coloring problems on $G$ are equivalent. As $G$ has no $L'$-coloring and is list $4$-critical, it follows that $G$ is $(H,L)$-minimal. 
    Therefore, by Theorem~\ref{thm:HLcrit}, $\rho_h(G) =  8n  - 5 m \leq -1$. Rearranging, $m \geq  \frac 85 n + \frac 15$.
\end{proof}

\bibliographystyle{abbrv}
\small
\bibliography{ref}
 
  \end{document}